\documentclass[english]{article}
\usepackage[T1]{fontenc}
\usepackage[latin9]{inputenc}
\usepackage{units}
\usepackage{amsthm}
\usepackage{amsmath}
\usepackage{amssymb}
\usepackage{stmaryrd}
\usepackage{stackrel}
\usepackage{esint}
\usepackage{xargs}[2008/03/08]
\usepackage[all]{xy}

\makeatletter
\numberwithin{equation}{section}
\numberwithin{figure}{section}
\theoremstyle{plain}
\newtheorem{thm}{\protect\theoremname}[section]
  \theoremstyle{definition}
  \newtheorem{defn}[thm]{\protect\definitionname}
  \theoremstyle{remark}
  \newtheorem{rem}[thm]{\protect\remarkname}
  \theoremstyle{plain}
  \newtheorem{cor}[thm]{\protect\corollaryname}
  \theoremstyle{plain}
  \newtheorem{lem}[thm]{\protect\lemmaname}
  \theoremstyle{plain}
  \newtheorem{prop}[thm]{\protect\propositionname}
  \theoremstyle{remark}
  \newtheorem*{rem*}{\protect\remarkname}
  \theoremstyle{plain}
  \newtheorem*{fact*}{\protect\factname}
  \theoremstyle{plain}
  \newtheorem*{lem*}{\protect\lemmaname}

\makeatother

\usepackage{babel}
  \providecommand{\corollaryname}{Corollary}
  \providecommand{\definitionname}{Definition}
  \providecommand{\factname}{Fact}
  \providecommand{\lemmaname}{Lemma}
  \providecommand{\propositionname}{Proposition}
  \providecommand{\remarkname}{Remark}
\providecommand{\theoremname}{Theorem}

\begin{document}
\global\long\def\pp#1{\frac{\partial}{\partial#1}}
\global\long\def\acc#1{\left\{  #1\right\}  }
\global\long\def\cro#1{\left[#1\right]}
\global\long\def\ppp#1#2{\frac{\partial#1}{\partial#2}}
\global\long\def\ddd#1#2{\frac{\mbox{d}#1}{\mbox{d}#2}}
\global\long\def\eps{\varepsilon}
\global\long\def\ww#1{\mathbb{#1}}
\global\long\def\norm#1{\left\Vert #1\right\Vert }
\global\long\def\abs#1{\left|#1\right|}
\global\long\def\cal#1{\mathcal{#1}}
\global\long\def\ps#1{\left\langle #1\right\rangle }
\global\long\def\crocro#1{\left\llbracket #1\right\rrbracket }
\global\long\def\form#1{\ww C\left\llbracket #1\right\rrbracket }
\global\long\def\tt#1{\mathtt{#1}}
\global\long\def\tx#1{\mathrm{#1}}
\newcommandx\fdiff[1][usedefault, addprefix=\global, 1=]{\tx{\widehat{Diff}}_{\tx{fib}}#1}
\global\long\def\diff#1{\widehat{\tx{Diff}}#1}
\newcommandx\sdiff[1][usedefault, addprefix=\global, 1=]{\tx{\widehat{Diff}}_{\omega}#1}
\global\long\def\sect#1#2#3{S\left(#1,#2,#3\right)}
\global\long\def\svf{\cal D_{\omega}}
\global\long\def\fvf{\cal D^{\left(1\right)}}
\newcommandx\fisot[2][usedefault, addprefix=\global, 1=\tx{fib}]{\widehat{\tx{Isot}}_{#1}\left(#2\right)}
\global\long\def\math#1{{\displaystyle #1}}
\global\long\def\pare#1{\Big({\displaystyle #1}\Big)}

\newcommandx\pain[1][usedefault, addprefix=\global, 1=j]{\left({\tt P_{#1}}\right)}
\global\long\def\sn{\widehat{\mathbb{\cal{SN}}}}
\global\long\def\sns{\widehat{\mathbb{\cal{SN}}}_{\omega}}
\global\long\def\snnd{\widehat{\mathbb{\cal{SN}}}_{\tx{nd}}}
\global\long\def\quotient#1#2{\bigslant{#1}{#2}}
\global\long\def\res#1{\tx{res}\left(#1\right)}

\newcommand{\bigslant}[2]{{\raisebox{.3em}{$#1$}\left/\raisebox{-.3em}{$#2$}\right.}}

\title{Doubly-resonant saddle-nodes in $\ww C^{3}$ and the fixed singularity
at infinity in the Painlevé equations: formal classification.}

\author{Amaury Bittmann%
\thanks{IRMA, Université de Strasbourg, 7 rue René Descartes, 67084 Strasbourg
Cedex, France\protect \\
bittmann@math.unistra.fr\protect \\
Office P-106: 03 68 85 0192%
}}
\maketitle
\begin{abstract}
In this work we consider formal singular vector fields in $\ww C^{3}$
with an isolated and doubly-resonant singularity of saddle-node type
at the origin. Such vector fields come from irregular two-dimensional
systems with two opposite non-zero eigenvalues, and appear for instance
when studying the irregular singularity at infinity in Painlevé equations
$\pain$, $j\in\acc{I,II,III,IV,V}$, for generic values of the parameters.
Under generic assumptions we give a complete formal classification
for the action of formal diffeomorphisms (by changes of coordinates)
fixing the origin and fibered in the independent variable $x$. We
also identify all formal isotropies (self-conjugacies) of the normal
forms. In the particular case where the flow preserves a transverse
symplectic structure, \emph{e.g.} for Painlevé equations, we prove
that the normalizing map can be chosen to preserve the transverse
symplectic form.
\end{abstract}

\section{Introduction}

\subsection{Definition and main result}

~

We consider singular vector fields $Y$ in $\ww C^{3}$ which can
be written in appropriate coordinates $\left(x,\mathbf{y}\right):=\left(x,y_{1},y_{2}\right)$
as 
\begin{eqnarray}
Y & = & x^{2}\pp x+\Big(-\lambda y_{1}+F_{1}\left(x,\mathbf{y}\right)\Big)\pp{y_{1}}+\Big(\lambda y_{2}+F_{2}\left(x,\mathbf{y}\right)\Big)\pp{y_{2}}\,\,\,\,\,,\label{eq: intro}
\end{eqnarray}
where $\lambda\in\ww C^{*}$ and $F_{1},\, F_{2}$ are formal power
series of order at least two. They represent singular irregular $2$-dimensional
systems having two opposite non-zero eigenvalues and a vanishing third
eigenvalue. 

Our main motivation is the study of the irregular singularity at infinity
in Painlevé equations $\pain$, $j\in\acc{I,II,III,IV,V}$, for generic
values of the parameters~\cite{MR816829}. These equations, discovered
(mainly) by Paul Painlevé~\cite{MR1554937}, share the property that
the only movable singularities of their solutions are poles (the so-called
Painlevé property); this is the complete list of all such equations
up to changes of variables. They have been intensively studied since
the important work of Okamoto~\cite{MR0458459}. The study of fixed
singularities, and more particularly those at infinity, started to
be investigated by Boutroux with his famous \emph{tritronquées} solutions
\cite{MR1554937}. Recently, several authors provided more complete
information about such singularities, studying the so-called \emph{quasi-linear
Stokes phenomena} and also giving connection formulas (\cite{MR1854431},
\cite{MR1219497} and \cite{MR2101932}). However, to the best of
our knowledge there are no general analytic classification for this
kind of doubly-resonant saddle-nodes yet (using normal form theory). 

More precisely, we would like to understand the action of germs of
analytic diffeomorphisms on such vector fields by changes of coordinates.
If one tries to do this, a first step would be to provide a formal
classification, that is to study the action of formal changes of coordinates
on these vector fields. This is the aim of this paper. Based on the
usual strategy employed for the classification of resonant vector
fields \cite{MR672182} in dimension $2$, we give in a forthcoming
paper a complete analytic classification for a more specific class
of vector fields, by studying the \emph{non-linear Stokes phenomena}.

To state our main results we need to introduce some notations and
nomenclature. 
\begin{itemize}
\item $\form{\mathbf{x}}$ is the $\ww C$-algebra of formal power series
in the (multi)variable $\mathbf{x}=\left(x_{1},\dots,x_{n}\right)$
with coefficients in $\ww C$. We denote by $\mathfrak{m}$ its unique
maximal ideal: it is formed by formal series with null constant term.
For any formal series $f_{1},\dots,f_{m}$ in $\form{\mathbf{x}}$,
we denote by $\ps{f_{1},\dots,f_{m}}$ the ideal generated by these
elements. 
\item $\fvf$ is the Lie algebra of formal vector fields at the origin of
$\ww C^{3}$ which are singular (\emph{i.e.} vanish at the origin).
Any formal vector field in $\fvf$ can be written 
\[
Y={\displaystyle b\pp x+b_{1}\pp{y_{1}}+b_{2}\pp{y_{2}}}
\]
with $b,b_{1},b_{2}\in\mathfrak{m}$. 
\item $\diff{}$ is the group of formal diffeomorphisms fixing the origin
of $\ww C^{3}$. It acts on $\fvf$ by conjugacy: if $\left(\Phi,Y\right)\in\diff{}\times\fvf$,
\begin{equation}
\Phi_{*}\left(Y\right):=\left(\mbox{D}\Phi\cdot Y\right)\circ\Phi^{-1}\qquad,\label{eq: push forward intro}
\end{equation}
where $\mbox{D}\Phi$ is the Jacobian matrix of $\Phi$.
\item $\fdiff$ is the subgroup of $\diff{}$ of diffeomorphisms fibered
in the $x$-coordinate, \emph{i.e. }of the form $\left(x,\mathbf{y}\right)\mapsto\left(x,\phi\left(x,\mathbf{y}\right)\right)$.\end{itemize}
\begin{defn}
\label{def: drsn}A \textbf{doubly-resonant saddle-node} is a vector
field $Y\in\fvf$ which is $\fdiff$ -conjugate to one of the form
\begin{eqnarray*}
Y & = & x^{2}\pp x+\Big(-\lambda y_{1}+F_{1}\left(x,\mathbf{y}\right)\Big)\pp{y_{1}}+\Big(\lambda y_{2}+F_{2}\left(x,\mathbf{y}\right)\Big)\pp{y_{2}}\,\,\,\,\,,
\end{eqnarray*}
with $\lambda\in\ww C^{*}$ and $F_{1},F_{2}\in\mathfrak{m}^{2}$.
We will denote by $\sn$ the set of all such formal vector fields.
\end{defn}
By Taylor expansion up to order $1$ with respect to $\mathbf{y}$,
given a vector field $Y\in\sn$ written as above we can consider the
associated 2-dimensional differential system:
\begin{equation}
x^{2}\ddd{\mathbf{y}}x=\mathbf{\alpha}\left(x\right)+\mathbf{A}\left(x\right)\mathbf{y}\left(x\right)+\mathbf{f}\left(x,\mathbf{y}\left(x\right)\right)\qquad,\label{eq: system doubly resonant saddle node}
\end{equation}
where ${\displaystyle \mathbf{y}}=\left(y_{1},y_{2}\right)$, such
that the following conditions hold: 
\begin{itemize}
\item ${\displaystyle \alpha\left(x\right)=\left(\begin{array}{c}
\alpha_{1}\left(x\right)\\
\alpha_{2}\left(x\right)
\end{array}\right)},$ with ${\displaystyle \alpha_{1},\alpha_{2}\in\ps x^{2}\subset\form x}$
\item $\mathbf{A}\left(x\right)\in\mbox{Mat}_{2,2}\left(\form x\right)$
with $\mathbf{A}\left(0\right)=\tx{Diag}\left(-\lambda,\lambda\right)$,
$\lambda\in\ww C^{*}$
\item ${\displaystyle \mathbf{f}\left(x,\mathbf{y}\right)=\left(\begin{array}{c}
f_{1}\left(x,\mathbf{y}\right)\\
f_{2}\left(x,\mathbf{y}\right)
\end{array}\right)}$, with $f_{1},f_{2}\in\ps{y_{1},y_{2}}^{2}\subset\form{x,\mathbf{y}}$.
\end{itemize}
Based on this expression, we state:
\begin{defn}
\label{def: non-deg} The \textbf{residue} of $Y\in\sn$ is the complex
number
\[
{\displaystyle \res Y:=\left(\frac{\mbox{Tr}\left(\mathbf{A}\left(x\right)\right)}{x}\right)_{\mid x=0}}\,\,\,\,.
\]
We say that $Y$ is\textbf{ non-degenerate }if $\res Y\in\ww C\backslash\ww Q_{\leq0}$,
and we denote by $\snnd\subset\sn$ the subset of non-degenerate vector
fields.
\end{defn}
We will prove in subsection \ref{sub: non degenerate invariant well defined}
that the residue of $Y\in\sn$ is invariant under the action of $\fdiff$
by conjugacy. We can state now our first main result.
\begin{thm}
\label{Th: Th drsn}Let $Y\in\snnd$ be a non-degenerate doubly-resonant
saddle-node. Then there exists a fibered diffeomorphism $\Phi\in\fdiff$
such that: 
\begin{eqnarray}
\Phi_{*}\left(Y\right) & = & x^{2}\pp x+\left(-\lambda+a_{1}x+c_{1}\left(v\right)\right)y_{1}\pp{y_{1}}\nonumber \\
 &  & +\left(\lambda+a_{2}x+c_{2}\left(v\right)\right)y_{2}\pp{y_{2}}\,\,\,\,,\label{eq: fibered normal form-1}
\end{eqnarray}
where we put $v:=y_{1}y_{2}$. Here, $c_{1},c_{2}$ in $\ps v=v\form v$
are formal power series with null constant term and $a_{1},a_{2}\in\ww C$
are such that $a_{1}+a_{2}=\res Y$.\end{thm}
\begin{rem}
We will see in Corollary \ref{cor: uniqueness tangent to id} and
Proposition \ref{prop: fibered isotropies} that $\Phi$ as above
is essentially unique (that is, unique up to pre-composition by linear
transforms).\end{rem}
\begin{defn}
The \textbf{parameter space }for $\snnd$ is the set 
\begin{eqnarray*}
\cal P & := & \acc{\mathbf{p}=\left(\lambda,a_{1},a_{2},c_{1},c_{2}\right)\in\ww C^{*}\times\left(\ww C^{2}\backslash\Delta\right)\times\left(v\form v\right)^{2}}
\end{eqnarray*}
where 
\begin{eqnarray*}
\Delta & = & \acc{\left(a_{1},a_{2}\right)\in\ww C^{2}\mid a_{1}+a_{2}\in\ww Q_{\leq0}}
\end{eqnarray*}
 is the locus of degeneracy. A vector field in the form $\left(\mbox{\ref{eq: fibered normal form-1}}\right)$
will be called a \textbf{normal form} of $\snnd$ with parameters
$\left(\lambda,a_{1},a_{2},c_{1},c_{2}\right)$ in $\cal P$.
\end{defn}
Let us consider the quotient space 
\[
\quotient{\cal P}{\left(\ww C^{*}\times\nicefrac{\ww Z}{2\ww Z}\right)}
\]
where the group ${\displaystyle \left(\ww C^{*}\times\nicefrac{\ww Z}{2\ww Z}\right)}$
acts on $\cal P$ as follows. Given $\mathbf{p}=\left(\lambda,a_{1},a_{2},c_{1},c_{2}\right)\in\cal P$,
$\theta\in\ww C^{*}$ and $\epsilon\in\nicefrac{\ww Z}{2\ww Z}$ we
define 
\begin{eqnarray*}
\theta\cdot\left(\lambda,a_{1},a_{2},c_{1},c_{2}\right) & = & \left(\lambda,a_{1},a_{2},c_{1}\circ\varphi_{\theta},c_{2}\circ\varphi_{\theta}\right)\\
\epsilon\cdot\left(\lambda,a_{1},a_{2},c_{1},c_{2}\right) & = & \begin{cases}
\left(\lambda,a_{1},a_{2},c_{1},c_{2}\right) & ,\,\mbox{if }\epsilon=0\\
\left(-\lambda,a_{2},a_{1},c_{2},c_{1}\right) & ,\,\mbox{if }\epsilon=1
\end{cases}\qquad,
\end{eqnarray*}
where $\varphi_{\theta}$ is the homothecy $v\mapsto\theta v$. If
two parameters $\mathbf{p},\mathbf{p}'\in\cal P$ are in the same
orbit for this action we write $\mathbf{p}\sim\mathbf{p}'$. Our second
main result shows the uniqueness of the normal forms up to this action.
\begin{thm}
\label{thm: uniqueness}Suppose $Z$ and $Z'$ are two normal forms
of $\snnd$ with respective parameters $\mathbf{p}=\left(\lambda,a_{1},a_{2},c_{1},c_{2}\right)\in\cal P$
and $\mathbf{p}'=\left(\lambda',a'_{1},a'_{2},c'_{1},c'_{2}\right)\in\cal P$.
Then $Z$ and $Z'$ are $\fdiff$-conjugate if and only $\mathbf{p}\sim\mathbf{p}'$.
\end{thm}
One can rephrase the above results in terms of group actions as follows.
\begin{cor}
There exists a bijection 
\begin{eqnarray*}
{\displaystyle \quotient{\snnd}{\fdiff}} & \simeq & \quotient{\cal P}{\left(\ww C^{*}\times\nicefrac{\ww Z}{2\ww Z}\right)}\,\,\,\,\,\,\,\,,
\end{eqnarray*}
where $\fdiff$ acts on $\snnd$ by conjugacy.
\end{cor}
Let us make some remarks.
\begin{rem}
~
\begin{enumerate}
\item The condition of non-degeneracy is necessary to obtain such normal
forms. For instance for any $a_{1},a_{2}\in\ww C$ such that $a_{1}+a_{2}=-\frac{p}{q}\in\ww Q_{\leq0}$,
with $\left(p,q\right)\in\ww N\times\ww N^{*}$, the vector field
\[
Y=x^{2}\pp x+\left(-\lambda+a_{1}x+x^{p+1}\left(y_{1}y_{2}\right)^{q}\right)y_{1}\pp{y_{1}}+\left(\lambda+a_{2}x\right)y_{2}\pp{y_{2}}\,\,,
\]
with residue ${\displaystyle \res Y=-\frac{p}{q}}$ is not $\fdiff$-conjugate
to a normal form as in Theorem \ref{Th: Th drsn}. Indeed, the resonant
term $x^{p+1}\left(y_{1}y_{2}\right)^{q}$ cannot be eliminated by
the action of $\fdiff.$
\item Notice that the above two results are not immediate consequences of
Poincaré-Dulac normal form theory. In fact, the usual Poincaré-Dulac
normal form possibly contains several additional resonant terms of
the form $\left(x^{k}\left(y_{1}y_{2}\right)^{l}\right)_{k,l\in\ww N}$,
and is far from being unique.
\end{enumerate}
\end{rem}

\subsection{Painlevé equations and the transversally symplectic case}

In \cite{MR816829} Yoshida shows that a vector field in the class
$\snnd$ naturally appears after a suitable compactification (given
by the so-called Boutroux coordinates \cite{Boutroux1913}) of the
phase-space of Painlevé equations $\pain$, $j\in\acc{I,II,III,IV,V}$
(for generic values of the parameters). In these cases the vector
field presents an additional Hamiltonian structure that will interest
us. 

Let us illustrate these computations in the case of the first Painlevé
equation: 
\begin{eqnarray*}
\left(P_{I}\right)\,\,\,\,\,\,\,\,\,\,\,\,\,\,\,\,\,\,\,\,\,\ddd{^{2}z_{1}}{t^{2}} & = & 6z_{1}^{2}+t\,\,\,\,\,\,\,\,\,\,\,\,\,\,.
\end{eqnarray*}
As is well known since Okamoto \cite{MR581468}, $\left(P_{I}\right)$
can be seen as a non-autonomous Hamiltonian system 
\[
\begin{cases}
\ppp{z_{1}}t=-\ppp H{z_{2}}\\
\ppp{z_{2}}t=\ppp H{z_{1}}
\end{cases}
\]
with Hamiltonian 
\begin{eqnarray*}
H\left(t,z_{1},z_{2}\right) & := & 2z_{1}^{3}+tz_{1}-\frac{z_{2}^{2}}{2}.
\end{eqnarray*}
More precisely, if we consider the standard symplectic form $\omega_{st}:=dz_{1}\wedge dz_{2}$
and the vector field
\begin{eqnarray*}
Z & := & \pp t-\ppp H{z_{2}}\pp{z_{1}}+\ppp H{z_{1}}\pp{z_{2}}
\end{eqnarray*}
induced by $\left(P_{I}\right)$, then the Lie derivative 
\[
\cal L_{Z}\left(\omega_{st}\right)=\left(\ppp{^{2}H}{t\partial z_{1}}\mbox{d}z_{1}+\ppp{^{2}H}{t\partial z_{2}}\mbox{d}z_{2}\right)\wedge\mbox{d}t=\mbox{d}z_{1}\wedge\mbox{d}t
\]
belongs to the ideal $\ps{\tx dt}$ generated by $\tx dt$ in the
exterior algebra $\Omega^{*}\left(\ww C^{3}\right)$ of differential
forms in variables $\left(t,z_{1},z_{2}\right)$. Equivalently, for
any $t_{1},t_{2}\in\ww C$ the flow of $Z$ at time $\left(t_{2}-t_{1}\right)$
acts as a symplectomorphism between fibers $\acc{t=t_{1}}$ and $\acc{t=t_{2}}$. 

The weighted compactification given by the Boutroux coordinates~\cite{MR1509163}
(see also \cite{Chiba}) defines a chart near $\acc{t=\infty}$ as
follows: 
\[
\begin{cases}
z_{2}=y_{2}x^{-\frac{3}{5}}\\
z_{1}=y_{1}x^{-\frac{2}{5}}\\
t=x^{-\frac{4}{5}}
\end{cases}\,\,\,.
\]
In the coordinates $\left(x,y_{1},y_{2}\right)$, the vector field
$Z$ is transformed, up to a translation $y_{1}\leftarrow y_{1}+\zeta$
with $\zeta=\frac{i}{\sqrt{6}}$, into the vector field 
\begin{eqnarray}
\tilde{Z} & = & -\frac{5}{4x^{\frac{1}{5}}}Y\label{eq: P1 ham at infinity}
\end{eqnarray}
where
\begin{eqnarray}
Y & = & x^{2}\pp x+\left(-\frac{4}{5}y_{2}+\frac{2}{5}xy_{1}+\frac{2\zeta}{5}x\right)\pp{y_{1}}\nonumber \\
 &  & +\left(-\frac{24}{5}y_{1}^{2}-\frac{48\zeta}{5}y_{1}+\frac{3}{5}xy_{2}\right)\pp{y_{2}}\,\,\,.\label{eq: P1 at infinity}
\end{eqnarray}
We observe that $Y$ is a\emph{ non-degenerate doubly-resonant saddle-node}
$Y$ as in Definitions \ref{def: drsn} and \ref{def: non-deg} with
residue $\res Y=1$. Furthermore we have: 
\[
\begin{cases}
\tx dt & =-\frac{4}{5}5^{\frac{4}{5}}x^{-\frac{9}{5}}\tx dx\\
\tx dz_{1}\wedge\tx dz_{2} & =\frac{1}{x}\left(\tx dy_{1}\wedge\tx dy_{2}\right)+\frac{1}{5x^{2}}\left(2y_{1}\tx dy_{2}-3y_{2}\tx dy_{1}\right)\wedge\tx dx\\
 & \in\,\frac{1}{x}\left(\tx dy_{1}\wedge\tx dy_{2}\right)+\ps{\tx dx}
\end{cases}\,\,\,\,\,,
\]
where $\ps{\mbox{d}x}$ denotes the ideal generated by $\mbox{d}x$.
We finally obtain 
\[
\begin{cases}
{\displaystyle \cal L_{Y}\left(\frac{\tx dy_{1}\wedge\tx dy_{2}}{x}\right)}=\frac{1}{5x}\left(3y_{2}\mbox{d}y_{1}-\left(2\zeta+2y_{1}\right)\mbox{d}y_{2}\right)\wedge\mbox{d}x\\
\cal L_{Y}\left(\mbox{d}x\right)=2x\mbox{d}x
\end{cases}\quad.
\]
 Therefore, both $x{\displaystyle \cal L_{Y}\left(\frac{\tx dy_{1}\wedge\tx dy_{2}}{x}\right)}$
and $\cal L_{Y}\left(\mbox{d}x\right)$ are differential forms which
lie in the ideal $\ps{\tx dx}$. This motivates the following definition.
\begin{defn}
\label{def: intro}Consider the rational 1-form 
\begin{eqnarray*}
\omega & := & \frac{\mbox{d}y_{1}\wedge\mbox{d}y_{2}}{x}~.
\end{eqnarray*}
We say that a formal vector field $Y\in\fvf$ is \textbf{transversally
Hamiltonian }(with respect to $\omega$ and $\tx dx$) if 
\begin{eqnarray*}
\mathcal{L}_{Y}\left(\tx dx\right)\in\left\langle \tx dx\right\rangle  & \mbox{ and } & x\mathcal{L}_{Y}\left(\omega\right)\in\left\langle \tx dx\right\rangle \qquad.
\end{eqnarray*}
We say a formal diffeomorphism $\Phi\in\diff{}$ is \textbf{transversally
symplectic} (with respect to $\omega$ and $\tx dx$) if
\[
\Phi^{*}\left(x\right)=x\mbox{ and }x\Phi^{*}\left(\omega\right)\in\, x\omega+\ps{\tx dx}\qquad.
\]
 (Here $\Phi^{*}\left(\omega\right)$ denotes the pull-back of $\omega$
by $\Phi$.)

We denote respectively by $\svf$ and $\sdiff$ the sets of transversally
Hamiltonian vector fields and transversally symplectic diffeomorphisms. \end{defn}
\begin{rem}
~
\begin{itemize}
\item The flow of a transversally Hamiltonian $X$ defines a map between
fibers $\acc{x=x_{1}}$ and $\acc{x=x_{2}}$ which sends $\omega_{\mid x=x_{1}}$
onto $\omega_{\mid x=x_{2}}$, since 
\[
\left(\exp\left(X\right)\right)^{*}\left(\omega\right)\in\,\omega+\ps{\mbox{d}x}\,\,.
\]

\item By our definition, a transversally symplectic diffeomorphism $\Phi\in\sdiff$
is necessarily a fibered diffeomorphism. In other words: $\sdiff\subset\fdiff$.
\end{itemize}
\end{rem}
\begin{defn}
\label{def: nc dr th}A \textbf{transversally Hamiltonian doubly-resonant
saddle-node} is a vector field $Y\in\cal D_{\omega}$ which is $\sdiff$
-conjugate to one of the form 
\begin{eqnarray*}
Y & = & x^{2}\pp x+\Big(-\lambda y_{1}+F_{1}\left(x,\mathbf{y}\right)\Big)\pp{y_{1}}+\Big(\lambda y_{2}+F_{2}\left(x,\mathbf{y}\right)\Big)\pp{y_{2}}\,\,\,\,\,,
\end{eqnarray*}
with $\lambda\in\ww C^{*}$ and $F_{1},F_{2}\in\mathfrak{m}^{2}$.
We will denote by $\sns$ the set of all such formal vector fields.
\end{defn}
Notice that a transversally Hamiltonian doubly-resonant saddle-node
$Y\in\sns$ is necessarily non-degenerate since its residue is always
equal to $1$. In other words: $\sns\subset\snnd$.
\begin{thm}
\label{thm: Th ham}Let $Y\in\sns$ be a transversally Hamiltonian
doubly-resonant saddle-node. Then, there exists a transversally symplectic
diffeomorphism $\Phi\in\sdiff$ such that: 
\begin{eqnarray}
\Phi_{*}\left(Y\right)= & x^{2}\pp x+\left(-\lambda+a_{1}x-c\left(v\right)\right)y_{1}\pp{y_{1}}+\left(\lambda+a_{2}x+c\left(v\right)\right)y_{2}\pp{y_{2}}\,\,\,\,.\label{eq: fibered normal form-1-1}
\end{eqnarray}
where we put $v:=y_{1}y_{2}$. Here, $c\left(v\right)$ in $v\form v$
is a formal power series with null constant term and $a_{1},a_{2}\in\ww C$
are such that $a_{1}+a_{2}=\res Y=1$. Furthermore this normal form
is unique with respect to the action of $\sdiff$.
\end{thm}
One can rephrase the theorem above in terms of group action.
\begin{cor}
There exists a bijection 
\begin{eqnarray*}
\quotient{\sns}{\sdiff} & \simeq & \ww C^{*}\times\acc{\left(a_{1},a_{2}\right)\in\ww C^{2}\mid a_{1}+a_{2}=1}\times v\form v\qquad.
\end{eqnarray*}
\end{cor}
\begin{rem}
~
\begin{enumerate}
\item As for Theorem \ref{Th: Th drsn}, $\Phi$ is essentially unique (Corollary
\ref{cor: uniqueness tangent to id}). This is an immediate consequence
of Theorem \ref{thm: uniqueness}. However, the fact that the normalizing
diffeomorphism $\Phi$ in Theorem \ref{thm: Th ham} is transversally
symplectic is not an immediate consequence of Theorem \ref{thm: uniqueness}.
\item The above normalization theorem can be interpreted as defining local\emph{
action-angle coordinates} for vector fields in $\sns$. More precisely,
if we consider the successive symplectic changes of coordinates 
\[
{\displaystyle \begin{cases}
y_{1}=\frac{e^{i\frac{\pi}{4}}}{\sqrt{2}}\left(u_{1}+iu_{2}\right)\\
y_{2}=\frac{e^{i\frac{\pi}{4}}}{\sqrt{2}}\left(u_{1}-iu_{2}\right)
\end{cases}}
\]
and 
\[
{\displaystyle \begin{cases}
u_{1}=\sqrt{2\rho}\cos\varphi\\
u_{2}=\sqrt{2\rho}\sin\varphi
\end{cases}}\;,
\]
then the vector field $\left(\mbox{\ref{eq: fibered normal form-1-1}}\right)$
becomes: 
\[
x^{2}\pp x+e^{-i\frac{\pi}{4}}x\sqrt{\rho}\pp{\rho}+i\left(\lambda+c\left(i\rho\right)+\frac{\left(a_{2}-a_{1}\right)}{2}x\right)\pp{\varphi}\,\,\,\,.
\]
Notice that the corresponding differential equation can be explicitly
integrated by quadratures in terms of an anti-derivative of $c$.
\end{enumerate}
\end{rem}
We will explain in section 4 how to compute inductively any finite
jet of $c\left(v\right)$ in the case of the Painlevé equations (for
which $c\left(v\right)$ is a germ of an analytic function at the
origin). 
\begin{cor}
Let $Y$ be as in (\ref{eq: P1 at infinity}). Then $a_{1}=a_{2}=\frac{1}{2}$,
$\lambda=\frac{8\sqrt{3\zeta}}{5}=\frac{4\cdot2^{\frac{3}{4}}\cdot3^{\frac{1}{4}}}{5}e^{\frac{i\pi}{4}}$
and 
\begin{eqnarray*}
c\left(v\right) & = & 3v+\left(9+\frac{167\cdot2^{\frac{1}{4}}\cdot3^{\frac{3}{4}}}{96}e^{\frac{3i\pi}{4}}\right)v^{2}\\
 &  & +\left(16+\frac{31837\sqrt{6}}{6912}i+\frac{5}{2}\cdot2^{\frac{1}{4}}\cdot3^{\frac{1}{4}}\cdot e^{\frac{3i\pi}{4}}\right)v^{3}+O\left(v^{4}\right)~.
\end{eqnarray*}

\end{cor}

\subsection{Known results }

~

In \cite{MR816829}, \cite{MR794761} Yoshida shows that the doubly-resonant
saddle-nodes arising from the compactification of Painlevé equations
$\pain$, $j\in\acc{I,II,III,IV,V}$ (for generic values for the parameters)
is conjugate to polynomial vector fields of the form 
\begin{eqnarray}
 & Z= & x^{2}\pp x+\pare{-\left(1+\gamma y_{1}y_{2}\right)+a_{1}x}y_{1}\pp{y_{1}}\nonumber \\
 &  & +\pare{1+\gamma y_{1}y_{2}+a_{2}x}y_{2}\pp{y_{2}}\,\,\,\,\,,\label{eq: normal form Yoshida}
\end{eqnarray}
with $\gamma\in\ww C^{*}$ and $\left(a_{1},a_{2}\right)\in\ww C^{2}$
such that $a_{1}+a_{2}=1$. One drawback of this result is that Yoshida
admits fibered transformations $\Psi\left(x,\mathbf{y}\right)=\left(x,\psi_{1}\left(x,\mathbf{y}\right),\psi_{2}\left(x,\mathbf{y}\right)\right)$
of a more general form: 
\begin{eqnarray}
\psi_{i}\left(x,\mathbf{y}\right) & = & y_{i}\left(1+\sum_{\substack{\left(k_{0},k_{1},k_{2}\right)\in\ww N^{3}\\
k_{1}+k_{2}\geq1
}
}\frac{q_{i,\mathbf{k}}\left(x\right)}{x^{k_{0}}}y_{1}^{k_{1}+k_{0}}y_{2}^{k_{1}+k_{0}}\right)\,\,\,\,\,,\label{eq: Yoshida}
\end{eqnarray}
where each $q_{i,\mathbf{k}}$ is a formal power series. Notice that
$x$ can appear with negative exponents and therefore $\Psi\notin\diff{}$.
As we will see in the next subsection, the problem (when seen from
the viewpoint of analytic classification) is that the transformations
used by Yoshida have ``small'' regions of convergence, in the sense
that one cannot cover an entire neighborhood of the origin in $\ww C^{3}$
by taking the union of these regions. On the contrary, we prove in
an upcoming work that the formal normalizations presented here can
be embodied by diffeomorphisms analytic on finitely many sectors whose
union is a neighborhood of the origin. This entails the classical
theory of summability of formal power series.

\subsection{Analytic results}

Several authors studied the problem of convergence of the conjugating
transformations described above. Some results (that we recall soon)
will hold not only in the class of formal objects, but also for Gevrey
(and even summable) ones, or more generally for holomorphic functions
with asymptotic expansions in sectorial domains. We refer to \cite{MR1346201}
and \cite{MR672182} for details on asymptotic expansions, Gevrey
and summability theory.

Assuming that the initial vector field is analytic, Yoshida proves
in \cite{MR794761} that he can chose a conjugacy of the form $\left(\mbox{\ref{eq: Yoshida}}\right)$
which is the asymptotic expansion of an analytic function in a domain
\begin{eqnarray*}
 & \Big\{\left(x,\mathbf{z}\right)\in\text{\ensuremath{S\times\mathbf{D}\left(0,\mathbf{r}\right)\mid\abs{z_{1}z_{2}}<\nu\abs x}}\Big\}
\end{eqnarray*}
for some small $\nu>0$, where $S$ is a sector of opening less than
$\pi$ with vertex at the origin and $\mathbf{D}\left(0,\mathbf{r}\right)$
is a polydisc of small poly-radius $\mathbf{r}=\left(r_{1},r_{2}\right)$.
Moreover, the $\left(q_{i,\mathbf{k}}\left(x\right)\right)_{i,\mathbf{k}}$
are in fact Gevrey-1 series.

Under more restrictive conditions (which correspond to $c_{1}=c_{2}=0$
and $\mbox{Re}\left(a_{1}+a_{2}\right)>0$ in Theorem \ref{Th: Th drsn}),
Shimomura, improving on a result by Iwano \cite{MR601918}, shows
in \cite{MR748020} that analytic doubly-resonant saddle-nodes satisfying
these conditions are conjugate to:
\begin{eqnarray*}
 &  & x^{2}\pp x+\left(-\lambda+a_{1}x\right)y_{1}\pp{y_{1}}+\left(\lambda+a_{2}x\right)y_{2}\pp{y_{2}}\,\,\,\,
\end{eqnarray*}
\emph{via} a diffeomorphism whose coefficients have asymptotic expansions
as $x\rightarrow0$ in sectors of opening greater than $\pi$. Stolovitch
generalized this result for any dimension in \cite{MR1386227}. Unfortunately,
as shown by Yoshida in \cite{MR816829}, the hypothesis $c_{1}=c_{2}=0$
is not met in the case of Painlevé equations. 

In a forthcoming series of papers we will prove an analytic version
of Theorem \ref{thm: Th ham}, valid in sectorial domains with sufficiently
large opening, which in turn will help us to provide an analytic classification.
Let us insist once more on the key fact that the union of these sectorial
domains forms a whole neighborhood of the origin.

\subsection{Outline of the paper}
\begin{itemize}
\item In section 2 we recall some basic concepts and results from the theory
of formal vector fields and differential forms.
\item In section 3 we prove Theorems \ref{Th: Th drsn}, \ref{thm: uniqueness}
and \ref{thm: Th ham}, and compute the isotropies of the associated
normal forms.
\item In section 4 we explain how to compute any finite jet of the formal
invariant $c$ in Theorem \ref{thm: Th ham} in the case of the Painlevé
equations.
\end{itemize}
\tableofcontents{}

\section{Background\label{sub: notations}}

We refer the reader to \cite{MR2363178}, \cite{MR647488} and \cite{Brjuno}
for a detailed introduction to formal vector fields and formal diffeomorphisms.
Although these concepts are well-known by specialists, we will recall
briefly the needed results and nomenclature.

\subsection{Formal power series, vector fields and diffeomorphisms}

~

As usual, we will denote a formal power series as ${\displaystyle f\left(\mathbf{x}\right)=\sum_{\mathbf{k}}f_{\mathbf{k}}\mathbf{x^{k}}}$
where, for all $\mathbf{k}=\left(k_{1},\dots,k_{n}\right)\in\ww N$,
$f_{\mathbf{k}}\in\ww C$ and $\mathbf{x^{k}}=x_{1}^{k_{1}}\dots x_{n}^{k_{n}}$.
We will also use the notation $\abs{\mathbf{k}}=k_{1}+\dots+k_{n}$
for the \emph{degree} of a \emph{monomial} $\mathbf{x^{k}}$ (which
is of \emph{homogenous degree }$\mathbf{k}=\left(k_{1},\dots,k_{n}\right)$). 

\bigskip{}

We denote respectively by $\form{\mathbf{x}}$, $\cal D$, $\diff{}$
the sets of formal power series (equipped with an algebra structure),
vector field (equipped with a Lie algebra structure), diffeomorphisms
(equipped with a group structure). The maximal ideal of the algebra
$\form{\mathbf{x}}$ formed by formal power series with null constant
term is denoted by $\mathfrak{m}$. 

\bigskip{}

A vector field will be seen either as a an element of $\left(\form{\mathbf{x}}\right)^{n}$
or as a derivation on $\form{\mathbf{x}}$: for any vector field 
\begin{eqnarray}
X & = & \alpha_{1}\pp{x_{1}}+\cdots+\alpha_{n}\pp{x_{n}}\in\left(\form{\mathbf{x}}\right)^{n}\,\,,\label{eq: fvf-1}
\end{eqnarray}
its \emph{Lie derivative }is defined as the operator 
\begin{eqnarray}
\cal L_{X}\left(f\right) & = & \alpha_{1}\ppp f{x_{1}}+\cdots+\alpha_{n}\ppp f{x_{n}}\,\,,\label{eq: lie deriv}
\end{eqnarray}
for any formal power series $f\in$$\form{\mathbf{x}}$. The Lie bracket
$\cro{X,Y}$ of two vector fields $X,Y\in\cal D$ is defined by
\begin{eqnarray*}
\cal L_{\cro{X,Y}} & \left(f\right)= & \cal L_{X}\left({\cal L_{Y}}\left(f\right)\right)-\cal L_{Y}\left({\cal L}_{X}\left(f\right)\right)
\end{eqnarray*}
for all $f\in\form{\mathbf{x}}$. 

\bigskip{}

Similarly, a formal diffeomorphism will be seen either as an element
of $\Phi\left(\mathbf{x}\right)\in\left(\form{\mathbf{x}}\right)^{n}$
such that $\Phi\left(\mathbf{0}\right)=\mathbf{0}$ and $\tx D_{\mathbf{0}}\Phi=\tx{Jac}\left(\Phi\left(\mathbf{0}\right)\right)\in\tx{Gl}_{n}\left(\ww C\right)$,
or as an algebra automorphism of $\form{\mathbf{x}}$: given a formal
series ${\displaystyle f={\displaystyle \sum_{\mathbf{k}\in\ww N^{n}}a_{\mathbf{k}}\mathbf{x^{k}}}\in\form{\mathbf{x}}}$,
we denote by 
\begin{equation}
\Phi^{*}\left(f\right)=\sum_{\mathbf{k}\in\ww N^{n}}a_{\mathbf{k}}\phi_{1}^{k_{1}}\cdots\phi_{n}^{k_{n}}\,\,,\label{eq: pullback-1}
\end{equation}
the \emph{pull-back} of $f$ by $\Phi\in\diff{}$, where 
\[
\phi_{1}=\Phi\left(x_{1}\right),\dots,\phi_{n}=\Phi\left(x_{n}\right)\,\,.
\]
The Jacobian matrix (or the linear part) of $\Phi$ in the basis $\left(x_{1},\dots,x_{n}\right)$
is the matrix ${\displaystyle \left(\ppp{\phi_{i}}{x_{j}}\left(0,\dots,0\right)\right)_{i,j}}$.

\bigskip{}

The \emph{order} $\mbox{ord}\left(f\right)$ (\emph{resp. $\mbox{ord}\left(X\right)$},
\emph{resp. $\mbox{ord}\left(\Phi\right)$}) of a non-zero formal
power series $f\in\form{\mathbf{x}}$ (\emph{resp.} vector field\emph{
$X\in\cal D$, resp.} diffeomorphism $\Phi\in\diff{}$) is the maximal
integer $k\geq0$ such that $f\in\mathfrak{m}^{k}$ (\emph{resp. $\cal L_{X}\left(\mathfrak{m}\right)\subset\mathfrak{m}^{k}$},
\emph{resp. $\Phi^{*}\left(\mathfrak{m}\right)\subset\mathfrak{m}^{k}$}).
The notion of order allows to define the classical \emph{Krull topology
}on\emph{ }$\form{\mathbf{x}}$, $\cal D$ and $\diff{}$. The set
of formal vector field of order at least $k$ is a submodule denoted
by $\cal D^{\left(k\right)}\subset\cal D$. In particular, $\cal D^{\left(1\right)}$
is the submodule of \emph{singular }formal vector fields. We denote
by $\cal A^{\left(k\right)}\subset\diff{}$ the normal subgroup formed
by those automorphisms $\Phi$ such that 
\[
\Phi\left(x_{i}\right)-x_{i}\in\mathfrak{m}^{k+1}
\]
for each $i=1,\dots,n$. Each element of $\cal A^{\left(k\right)}$
will be called a formal diffeomorphism \emph{tangent to the identity
}up to order $k$. 

\bigskip{}

Given a subgroup $\cal G\subset\diff{}$, we say that two vector fields
$Y_{1},~Y_{2}$ in $\cal D$ are $\cal G-$\emph{conjugate} if there
exits a $\Phi\text{\ensuremath{\in}}\cal G$ such that: 
\begin{eqnarray*}
\cal L_{Y_{1}}\circ\Phi & = & \Phi\circ\cal L_{Y_{2}}\quad.
\end{eqnarray*}

\bigskip{}

The following two lemmas will be important in the proof of Theorem
\ref{Th: Th drsn}.
\begin{lem}
\label{lem: order of Lie bracket}Let $X,Y\in\fvf$ be two singular
formal vector fields. Then: 
\begin{eqnarray*}
\tx{ord}\left(\cro{X,Y}\right) & \geq & \tx{ord}\left(X\right)+\tx{ord}\left(Y\right)-1\,\,\,.
\end{eqnarray*}

\end{lem}
\medskip{}

\begin{lem}
\label{lem: diffeo formel infinite composition}Let $\left(d_{n}\right)_{n\geq0}\subset\ww N_{>0}$
be a strictly increasing sequence of positive integers, and $\left(\Phi_{n}\right)_{n\geq0}$
a sequence of formal diffeomorphisms. Assume that for all $n\geq0$,
\[
\Phi_{n}\left(\mathbf{x}\right)=\mathbf{x}+P_{d_{n}}\left(\mathbf{x}\right)\,\,\left(\mbox{ mod }\mathfrak{m}^{d_{n}+1}\right)\,\,\,\,\,\,\,\,,
\]
where $P_{d_{n}}\left(\mathbf{x}\right)$ is a vector homogenous polynomial
of degree $d_{n}$. Then the sequence $\left(\Phi^{\cro n}\right)_{n\geq0}$,
defined by $\Phi^{\cro n}=\Phi_{n}\circ\dots\circ\Phi_{0}$, for all
$n\geq0$, is convergent, of limit $\Phi\in\diff{}$.

Moreover, if each $\Phi_{n}$ is fibered then $\Phi$ is fibered too.\end{lem}
\begin{proof}
It suffices to prove by induction that for all $n\geq0$: 
\[
\Phi_{n}\circ\dots\circ\Phi_{0}\left(\mathbf{x}\right)=\mathbf{x}+P_{d_{0}}\left(\mathbf{x}\right)+\dots+P_{d_{n}}\left(\mathbf{x}\right)\,\,\left(\mbox{ mod }\mathfrak{m}^{d_{n}+1}\right)\,\,,
\]
because the sequence $\left(d_{n}\right)_{n\geq0}\subset\ww N_{>0}$
is strictly increasing.
\end{proof}

\subsection{Exponential map and logarithm}

~

Given formal vector field $X\in\fvf$ and a formal power series $f\in\form{\mathbf{x}}$
and we set 
\[
\begin{cases}
\cal L_{X}^{\circ0}\left(f\right)=f\\
\cal L_{X}^{\circ\left(k+1\right)}\left(f\right):=\cal L_{X}\left(\cal L_{X}^{\circ k}\left(f\right)\right) & \,,\,\mbox{for all }k\geq0
\end{cases}
\]
so that we can consider the algebra homomorphism given by:
\begin{equation}
\tx{exp}\left(X\right)^{*}:\, f\mapsto\underset{k\geq0}{\sum}\frac{1}{k!}\cal L_{X}^{\circ k}\left(f\right)\quad.\label{eq: exp}
\end{equation}
This series is convergent in the Krull topology and defines in fact
a formal diffeomorphism, which is called the\emph{ time $1$ formal
flow} of $X\in\fvf$ or the \emph{exponential} of $X$. (see \emph{e.g.
}section 3 in \cite{MR2363178}).

\bigskip{}

For any vector field $X\in\fvf$, we consider also the \emph{adjoint
map} 
\begin{eqnarray*}
\mbox{\ensuremath{\tx{ad}}}_{X}:\,\fvf & \rightarrow & \fvf\\
Y & \mapsto & \cro{X,Y}
\end{eqnarray*}
and define 
\begin{eqnarray*}
\begin{cases}
\math{\left(\mbox{\ensuremath{\tx{ad}}}_{X}\right)^{\circ0}:=\mbox{\ensuremath{\tx{Id}}}}\\
\math{\left(\mbox{\ensuremath{\tx{ad}}}_{X}\right)^{\circ\left(k+1\right)}:=\mbox{\ensuremath{\tx{ad}}}_{X}\circ\left(\mbox{\ensuremath{\tx{ad}}}_{X}\right)^{\circ k}} & ,\,\forall k\in\ww N
\end{cases} &  & \,\,\,\,\,\,.
\end{eqnarray*}
We will need the following classical formula (see \cite{MR647488}).
\begin{prop}
\label{prop: exp (ad)}Given $X,Y\in\fvf$:
\[
\tx{exp}\left(X\right)_{*}\left(Y\right)=\tx{exp}\left(\tx{ad}_{X}\right)\left(Y\right)\quad,
\]
where 
\[
\math{\tx{exp}\left(\mbox{\ensuremath{\tx{ad}}}_{X}\right)\left(Y\right)=\underset{k\geq0}{\sum}\frac{1}{k!}\left(\mbox{\ensuremath{\tx{ad}}}_{X}\right)^{\circ k}\left(Y\right)=Y+\frac{1}{1!}\cro{X,Y}+\frac{1}{2!}\cro{X,\cro{X,Y}}+\dots}\quad.
\]

\end{prop}
We also recall the existence of a logarithm for all formal diffeomorphisms
tangent to the identity (see \cite{MR2363178}, section 3).
\begin{prop}
\label{prop: exp is onto}For any formal diffeomorphism $\Phi\in\diff{}$,
there exists a unique vector field $F\in\cal D^{\left(2\right)}$
such that $\Phi=\varphi\circ\exp\left(F\right)$, where $\varphi\in\diff{}$
is the linear change of coordinate given by $\mbox{D}_{0}\Phi$. Moreover,
for each $k\geq2$, the \emph{exponential map} defines a bijection
between $\cal D^{\left(k\right)}$ and $\cal A^{\left(k-1\right)}$.
\end{prop}

\subsection{Jordan decomposition and Dulac-Poincaré normal forms}

~

According to \cite{MR647488}, any singular formal vector field $X\in\fvf$
admits a unique \emph{Jordan decomposition}: 
\begin{equation}
X=X_{S}+X_{N},\mbox{ with }X_{S},X_{N}\in\fvf\mbox{ and }\cro{X_{S},X_{N}}=0\,\,\,\,,\label{eq: Jordan}
\end{equation}
where the restriction of $X_{S}$ $\left(resp.\, X_{N}\right)$ to
each $k$-jet vector space $\mbox{J}_{k}=\nicefrac{\mathfrak{m}}{\mathfrak{m}^{k}}$
(which is finite dimensional), $k\geq0$, is semi-simple $\left(resp.\mbox{ nilpotent}\right)$.
This decomposition is compatible with truncation and invariant by
conjugacy: if $X=X_{S}+X_{N}$ is the Jordan decomposition of $X$
then 
\begin{enumerate}
\item for all $k\geq0$, $\mbox{j}_{k}\left(X\right)=\mbox{j}_{k}\left(X_{S}\right)+\mbox{j}_{k}\left(X_{N}\right)$
is the Jordan decomposition of $\mbox{j}_{k}\left(X\right)$(here,
for any singular vector field $Y\in\fvf$, $\mbox{j}_{k}\left(Y\right)$
is the endomorphism $\mbox{J}_{k}\rightarrow\mbox{J}_{k}$ induced
by $\cal L_{Y}$);
\item for any formal diffeomorphism $\varphi\in\diff{}$, $\varphi_{*}\left(X\right)=\varphi_{*}\left(X_{S}\right)+\varphi_{*}\left(X_{N}\right)$
is the Jordan decomposition of $\varphi_{*}\left(X\right)$.\end{enumerate}
\begin{defn}
We say that $X\in\fvf$ is in \emph{Poincaré-Dulac normal form}\textbf{
}if its Jordan decomposition $X=X_{S}+X_{N}$ is such that $X_{S}$
is in diagonal form, \emph{i.e.} $X_{S}=S\left(\lambda\right)$, where
$\lambda:=\left(\lambda_{1},\dots,\lambda_{n}\right)\in\ww C^{n}$
and
\begin{eqnarray}
S\left(\lambda\right) & := & \lambda_{1}x_{1}\pp{x_{1}}+\dots+\lambda_{n}x_{n}\pp{x_{n}}\,\,\,\,.\label{eq: log notation}
\end{eqnarray}

\end{defn}
As mentioned in the introduction, according to Poincaré-Dulac Theorem,
any singular vector field is conjugate to a Poincaré-Dulac normal
form, but this normal form is far from being unique: every vector
field is conjugate to many distinct Poincaré-Dulac normal forms.
\begin{defn}
A \emph{monomial vector field} is a vector field in $\cal D$ of the
form $\mathbf{x^{k}}S\left(\mu\right)$ for some $\mathbf{k}\in\cal I$,
where $\cal I$ is the index set
\[
\cal I:=\acc{\mathbf{k}=\left(k_{1},\dots k_{n}\right)\in\left(\ww Z_{\geq-1}\right)^{n}\mid\mbox{ at most one of the }k_{j}\mbox{'s is }\mbox{-1}}\quad,
\]
and some $\mu\in\ww C^{n}$ with the condition that $\mu=\underset{\,\,\,\begin{array}{c}
\uparrow\\
j
\end{array}}{\left(0,\dots,0,\mu_{j},0,\dots0\right)}$ if $k_{j}=-1$.
\end{defn}
Fixing $\lambda\in\ww C^{n}$, each monomial vector field $\mathbf{x^{k}}S\left(\mu\right)$
is an eigenvector for $\tx{ad}_{S\left(\lambda\right)}$ with eigenvalue
\[
\ps{\lambda,\mathbf{k}}:=\lambda_{1}k_{1}+\dots+\lambda_{n}k_{n}\,\,\,\,.
\]
This is a consequence of the following elementary lemma.
\begin{lem}
\label{lem: Lie bracket in log notation}For all $\lambda,\mu\in\ww C^{n}$,
and for all $\mathbf{l},\mathbf{m}\in\ww Z^{n}$:
\begin{eqnarray*}
\left[\mathbf{x}^{\mathbf{l}}S\left(\lambda\right),\mathbf{x}^{\mathbf{m}}S\left(\mu\right)\right] & = & \mathbf{x}^{\mathbf{l}+\mathbf{m}}\left(\left\langle \lambda,\mathbf{m}\right\rangle S\left(\mu\right)-\left\langle \mu,\mathbf{l}\right\rangle S\left(\lambda\right)\right)\,\,\,\,.
\end{eqnarray*}
\end{lem}
\begin{rem}
\label{rem: monomial exp}Notice that each $X\in\cal D$ can be uniquely
written as an infinite sum of monomial vector fields of the form
\[
X=\sum_{\mathbf{k}\in\cal I}\mathbf{x}^{\mathbf{k}}S\left(\mu_{\mathbf{k}}\right)\quad,
\]
which is a Krull-convergent series in $\cal D$. We will call this
expression the \emph{monomial expansion} of $X$.
\end{rem}
Assume now that $X=S\left(\lambda\right)+X_{N}$ is in Poincaré-Dulac
normal form and let us consider the monomial expansion of $X_{N}:$
\begin{eqnarray*}
X_{N} & = & \sum_{\mathbf{k}\in\cal I}\mathbf{x^{k}}S\left(\mu_{\mathbf{k}}\right)\quad.
\end{eqnarray*}
The condition $\cro{X_{S},X_{N}}=0$ is equivalent to require 

\[
\forall\mathbf{k}\in\cal I,\,\ps{\lambda,\mathbf{k}}\neq0\Longrightarrow\mu_{\mathbf{k}}=0\quad;
\]
in other words, each $\mathbf{x^{k}}$ in the monomial expansion of
$X_{N}$ is a so-called \emph{resonant monomial}.
\begin{prop}
\label{prop: conjuation with resonant monomial}Let $X,Y\in\fvf$
be two vector fields in Poincaré-Dulac normal form with the same semi-simple
part $S\left(\mu\right)$ for some $\mu\in\ww C^{n}$, and with nilpotent
parts in $\cal D^{\left(2\right)}$: 
\[
\begin{cases}
X=S\left(\mu\right)+X_{N} & \mbox{ , with }X_{N}\in\cal D^{\left(2\right)},\mbox{ nilpotent, and }\cro{S\left(\mu\right),X_{N}}=0\\
Y=S\left(\mu\right)+Y_{N} & \mbox{ , with }Y_{N}\in\cal D^{\left(2\right)},\mbox{ nilpotent, and }\cro{S\left(\mu\right),Y_{N}}=0
\end{cases}\,\quad.
\]
Assume $X$ and $Y$ are conjugate by a formal diffeomorphism $\Phi$
such that $\mbox{D}_{0}\Phi=\tx{diag}\left(\lambda_{1},\dots,\lambda_{n}\right)$
for some $\lambda_{1},\dots,\lambda_{n}\in\ww C^{*}$. If we write
$\Phi=\varphi\circ\exp\left(F\right)$ for some vector field $F\in\cal D^{\left(2\right)}$,
where $\varphi\in\diff{}$ is the linear diffeomorphism associated
to $\mbox{D}_{0}\Phi=\tx{diag}\left(\lambda_{1},\dots,\lambda_{n}\right)$,
then necessarily $\cro{S\left(\mu\right),F}=0$.\end{prop}
\begin{rem}
Recall that the condition $\cro{S\left(\mu\right),F}=0$ means that
if we write ${\displaystyle F=\sum_{\mathbf{k}\in\cal I}\mathbf{x^{k}}S\left(\lambda_{\mathbf{k}}\right)}$,
then $\ps{\mu,\mathbf{k}}\neq0\Longrightarrow\lambda_{\mathbf{k}}=0$.\end{rem}
\begin{proof}
We can assume without loss of generality that $\Phi$ is tangent to
the identity. Indeed by setting $\math{P:=\left(\mbox{D}_{0}\Phi\right)^{-1}}$,
we obtain that $P\circ\Phi$ is tangent to the identity and conjugates
$X$ to $\tilde{Y}=\mbox{D}P\cdot\left(Y\circ P^{-1}\right)$. Since
$\mbox{D}P$ is diagonal, the assumptions made on $Y$ are also met
by $\tilde{Y}$. Moreover, it is clear that the property we have to
prove is true for $\Phi$ if and only if it is true for $P\circ\Phi$.
Therefore we may suppose that $\Phi$ is tangent to the identity.
According to Proposition \ref{prop: exp is onto}, there exists $F\in\cal D^{\left(2\right)}$
such that $\tx{exp}\left(F\right)=\Phi$, while according to Proposition
\ref{prop: exp (ad)} we have:

\begin{eqnarray}
\tx{exp}\left(F\right)_{*}\left(S\left(\mu\right)\right) & = & S\left(\mu\right)+\cro{F,S\left(\mu\right)}+\frac{1}{2!}\cro{F,\cro{F,S\left(\mu\right)}}+\dots
\end{eqnarray}
Since $\tx{exp}\left(F\right)_{*}\left(S\left(\mu\right)\right)=S\left(\mu\right)$
by uniqueness of the Jordan decomposition, we have 
\begin{eqnarray}
\cro{F,S\left(\mu\right)}+\frac{1}{2!}\cro{F,\cro{F,S\left(\mu\right)}}+\dots & = & 0\,\,\,\,.\label{eq: preuve lemme PD}
\end{eqnarray}
This implies that $\cro{F,S\left(\mu\right)}=0$, using Lemma \ref{lem: order of Lie bracket}
and the fact that $\mbox{ord}\left(F\right)\geq2$. \end{proof}
\begin{rem}
\label{rem: D is diagonal}The assumption that $\mbox{D}_{0}\Phi$
is in diagonal form necessarily holds if $\mu_{i}\neq\mu_{j}$, for
all $i\neq j$.
\end{rem}

\subsection{\label{sub:Formal-differential-forms}Formal differential forms}

~
\begin{defn}
We denote by $\widehat{\Omega^{1}}\left(\form{\mathbf{x}}\right)$
(or just $\widehat{\Omega^{1}}$ for simplicity) the set of \emph{formal
1-forms} in $\ww C^{n}$. It is the dual of $\mbox{Der}\left(\form{\mathbf{x}}\right)$
as $\form{\mathbf{x}}$-module. 
\end{defn}
Fixing the dual basis $\left(\mbox{d}x_{1},\dots\mbox{d}x_{n}\right)$
of $\left(\ww C^{n}\right)^{*}$, $\widehat{\Omega^{1}}\left(\form{\mathbf{x}}\right)$
is a free $\form{\mathbf{x}}$-module of rank $n$, generated by $\mbox{d}x_{1},\dots,\mbox{d}x_{n}$. 
\begin{defn}
For any $p\in\ww N$, we denote the $p$-exterior product of $\widehat{\Omega^{1}}\left(\form{\mathbf{x}}\right)$
by 
\[
\widehat{\Omega^{p}}\left(\form{\mathbf{x}}\right):=\bigwedge^{p}\widehat{\Omega^{1}}\left(\form{\mathbf{x}}\right)
\]
 (or just $\widehat{\Omega^{p}}$). Its elements will be called \emph{formal
$p$-forms}. 
\end{defn}
The set of $0$-forms is the set of formal series: $\widehat{\Omega^{0}}\left(\form{\mathbf{x}}\right):=\form{\mathbf{x}}$. 
\begin{defn}
We denote by 
\[
\widehat{\Omega}\left(\form{\mathbf{x}}\right):=\stackrel[p=0]{+\infty}{\oplus}\widehat{\Omega^{p}}\left(\form{\mathbf{x}}\right)
\]
 (or just $\widehat{\Omega}$ for simplicity) the exterior algebra
of the formal forms in $\ww C^{n}$, and by $\mbox{d}$ the exterior
derivative on it.
\end{defn}
One can also extend the Krull topology to $\widehat{\Omega}$.

We can define the action of $\diff{}$ by \emph{pull-back} on $\widehat{\Omega}\left(\form{\mathbf{x}}\right)$
thanks to the following properties: 
\begin{enumerate}
\item $\ww C$-linearity
\item for all $f\in\form{\mathbf{x}}$, $\Phi^{*}\left(f\right)$ is defined
as in (\ref{eq: pullback-1})
\item $\forall\alpha,\beta\in\widehat{\Omega}\left(\form{\mathbf{x}}\right)$,
$\forall\Phi\in\diff{}$, $\Phi^{*}\left(\alpha\wedge\beta\right)=\Phi^{*}\left(\alpha\right)\wedge\Phi^{*}\left(\beta\right)$
\item $\forall\Phi\in\diff{}$, $\Phi^{*}\circ\mbox{d}=\mbox{d}\circ\Phi^{*}$.
\end{enumerate}
For any $X\in\fvf$ and $\alpha\in\widehat{\Omega}\left(\form{\mathbf{x}}\right)$,
we denote by $\cal L_{X}\left(\alpha\right)$ the \emph{Lie derivative}
of $\alpha$ with respect to $X$. We recall that $\cal L_{X}$ is
uniquely determined by the following properties: 
\begin{enumerate}
\item for all $k\geq0$, $\cal L_{X}:\widehat{\Omega^{k}}\left(\form{\mathbf{x}}\right)\longrightarrow\widehat{\Omega^{k}}\left(\form{\mathbf{x}}\right)$
is linear
\item for all $f\in\form{\mathbf{x}}$ (\emph{i.e.} $f$ is a 0-form ),
$\cal L_{X}\left(f\right)$ is as in Definition (\ref{eq: lie deriv})
\item $\cal L_{X}$ is a derivation of $\widehat{\Omega}\left(\form{\mathbf{x}}\right)$,
\emph{i.e. }for all $\alpha,\beta\in\widehat{\Omega}\left(\form{\mathbf{x}}\right)$:
\[
\cal L_{X}\left(\alpha\wedge\beta\right)=\cal L_{X}\left(\alpha\right)\wedge\beta+\alpha\wedge\cal L_{X}\left(\beta\right)
\]
(Leibniz rule)
\item $\cal L_{X}\circ\mbox{d}=\mbox{d}\circ\cal L_{X}$.
\end{enumerate}
We will need the following classical formula, which extends $\left(\mbox{\ref{eq: exp}}\right)$.
\begin{prop}
\label{prop: exp }

$\forall\alpha\in\widehat{\Omega}\left(\form{\mathbf{x}}\right),\, X\in\fvf$:
\[
\tx{exp}\left(X\right)^{*}\left(\alpha\right)=\tx{exp}\left(\cal L_{X}\right)\left(\alpha\right)=\underset{k\geq0}{\sum}\frac{1}{k!}\cal L_{X}^{\circ k}\left(\alpha\right)\,\,\,.
\]
\end{prop}
\begin{proof}
(Sketch)

This formula is true for $0$-forms, and we just has to prove it for
$1$-forms, because we can then it extend to any $p$-form using the
exterior product and the Leibniz formula. In order to prove the result
for $1$-forms, one has to use the fact that $\cal L_{X}\circ\mbox{d}=\mbox{d}\circ\cal L_{X}$. 
\end{proof}
With the same arguments, and using formulas $\Phi^{*}\circ\mbox{d}=\mbox{d}\circ\Phi^{*}$
and $\Phi^{*}\left(\alpha\wedge\beta\right)=\Phi^{*}\left(\alpha\right)\wedge\Phi^{*}\left(\beta\right)$,
we can prove the following Proposition.
\begin{prop}
\label{prop: lien diffeo, chp vect et form}For all $\Phi\in\diff{}$,
$X\in\fvf$ and $\theta\in\widehat{\Omega}\left(\form{\mathbf{x}}\right)$,
we have; 
\[
\Phi^{*}\left(\cal L_{\Phi_{*}\left(X\right)}\left(\omega\right)\right)=\cal L_{X}\left(\Phi^{*}\left(\omega\right)\right)\,\,\,\,.
\]
In other words, the following diagram is commutative for all $p\geq0$:
\end{prop}
\[
\xymatrix{\widehat{\Omega^{p}}\ar[r]^{\Phi^{*}}\ar[d]_{\cal L_{\Phi_{\text{*}}\left(X\right)}} & \widehat{\Omega^{p}}\ar[d]^{{\cal L}_{X}}\\
\widehat{\Omega^{p}}\ar[r]_{\Phi^{*}} & \widehat{\Omega^{p}}
}
\]

From now on, we set $n=3$, we denote by $\mathbf{x}=\left(x,\mathbf{y}\right)=\left(x,y_{1},y_{2}\right)$
the coordinates in $\ww C^{3}$ . 
\begin{defn}
We denote by $\ps{\mbox{d}x}$ the ideal generated by $\mbox{d}x$
in $\widehat{\Omega}=\widehat{\Omega}\left(\form{x,\mathbf{y}}\right)$:
it is the set of forms $\theta\in\widehat{\Omega}$ such that $\theta=\mbox{d}x\wedge\eta$,
for some $\eta\in\widehat{\Omega}$.\end{defn}
\begin{prop}
\label{prop: exp symplectic} Let $\theta\in\widehat{\Omega}$, $X\in\fvf$
and set $\Phi:=\exp\left(X\right)\in\diff{}$. Then the following
assertions are equivalent:
\begin{enumerate}
\item \emph{$\cal L_{X}\left(x\right)=0$ }and $\cal L_{X}\left(\theta\right)\in\ps{\mbox{d}x}$
\item $\Phi^{*}\left(x\right)=x$ and $\Phi^{*}\left(\theta\right)\in\theta+\ps{\mbox{d}x}$.
\end{enumerate}
\end{prop}
\begin{proof}
It is just a consequence of Propositions \ref{prop: exp } and \ref{prop: exp (ad)}.
\end{proof}
The next Lemma is proved by induction, as Lemma \ref{lem: diffeo formel infinite composition}.
\begin{lem}
\label{lem: suite de diffeo symplectique}In the situation described
in Lemma \ref{lem: diffeo formel infinite composition}, if we further
assume that there exists a form $\theta\in\widehat{\Omega}$ such
that $\Phi_{n}^{*}\left(\theta\right)\in\theta+\ps{\mbox{d}x}$, for
all $n\geq0$, then $\Phi^{*}\left(\theta\right)\in\theta+\ps{\mbox{d}x}$.
\end{lem}

\subsection{\label{sub: transversally sn}Transversal Hamiltonian vector fields
and transversal symplectomorphisms}

~

We will need in fact to deal with forms with rational coefficients,
and more precisely with 
\begin{eqnarray*}
\omega & := & \frac{\mbox{d}y_{1}\wedge\mbox{d}y_{2}}{x}~.
\end{eqnarray*}
Given a formal vector field $X$ such that $\cal L_{X}\left(x\right)\in\ps x$
we can easily extend its Lie derivative action to the set ${\displaystyle x^{-1}\widehat{\Omega}\left(\form{x,\mathbf{y}}\right)}$
by setting: 
\begin{eqnarray*}
\cal L_{X}\left(\frac{1}{x}\theta\right) & := & -\frac{\cal L_{X}\left(x\right)}{x^{2}}\theta+\frac{1}{x}\cal L_{X}\left(\theta\right)\,,\,\theta\in\widehat{\Omega}\left(\form{x,\mathbf{y}}\right)\\
 & \in & x^{-1}\widehat{\Omega}\left(\form{x,\mathbf{y}}\right)\mbox{ , because }\cal L_{X}\left(x\right)\in\ps x\qquad.
\end{eqnarray*}
In particular we have
\[
x\cal L_{X}\left(\frac{1}{x}\theta\right)\in\widehat{\Omega}\left(\form{x,\mathbf{y}}\right)\quad.
\]
Notice that if a vector field $X$ satisfy $\cal L_{X}\left(\mbox{d}x\right)\in\ps{\mbox{d}x}$,
then $\cal L_{X}\left(x\right)\in\ps x$. Similarly, we naturally
extend the action of fibered diffeomorphisms by pull-back on rational
forms in $x^{-1}\widehat{\Omega}\left(\form{x,\mathbf{y}}\right)$
by: 
\[
\Phi^{\text{*}}\left(\frac{1}{x}\theta\right)=\frac{1}{x}\Phi^{*}\left(\theta\right),\,\mbox{for }\left(\Phi,\theta\right)\in\fdiff\times\widehat{\Omega}\left(\form{x,\mathbf{y}}\right)
\]
so that 
\[
x\Phi^{\text{*}}\left(\frac{1}{x}\theta\right)=\Phi^{*}\left(\theta\right)\qquad.
\]

Recalling Definition \ref{def: intro}, we can now state a result
analogous to Proposition \ref{prop: exp symplectic}.
\begin{prop}
\label{prop: expo de champ transv hamilt}Let $F\in\fvf$ be a singular
vector field. The following two statements are equivalent:
\begin{enumerate}
\item $\exp\left(F\right)$ is a transversally symplectic diffeomorphisms
,
\item $\cal L_{F}\left(x\right)=0$ and $F$ is transversally Hamiltonian.
\end{enumerate}
\end{prop}
\begin{proof}
It is just a consequence of Proposition \ref{prop: exp symplectic}.\end{proof}
\begin{lem}
\label{lem: push forward symplectic fibr=0000E9}Let $\Phi\in\sdiff$
and $X\in\cal D{}_{\omega}$ . Then, $\Phi_{*}\left(X\right)\in\cal D{}_{\omega}$. 
\begin{proof}
This comes from Proposition \ref{prop: lien diffeo, chp vect et form}:
$\Phi^{*}\left(\cal L_{\Phi_{*}\left(X\right)}\left(\omega\right)\right)=\cal L_{X}\left(\Phi^{*}\left(\omega\right)\right)$,
and from the fact that $\sdiff$ is a group, so $\Phi^{-1}\in\sdiff$
. Consequently we have: 
\begin{eqnarray*}
x\cal L_{\Phi_{*}\left(X\right)}\left(\omega\right) & = & x\left(\Phi^{-1}\right)^{*}\cal L_{X}\left(\Phi^{*}\left(\omega\right)\right)\\
 & = & x\left(\Phi^{-1}\right)^{*}\cal L_{X}\left(\omega+\ps{\mbox{d}x}\right)\\
 & = & x\left(\Phi^{-1}\right)^{*}\left(\cal L_{X}\left(\omega\right)\right)+x\left(\Phi^{-1}\right)^{*}\left(\cal L_{X}\left(\ps{\tx dx}\right)\right)\\
 & = & x\left(\Phi^{-1}\right)^{*}\left(\ps{\tx dx}\right)+x\left(\Phi^{-1}\right)^{*}\left(\ps{\tx dx}\right)\\
 & \in & \ps{\mbox{d}x}\,\,\,.
\end{eqnarray*}

\end{proof}
\end{lem}
\begin{rem}
\label{rem: notation saddle node symplectic}In other words, we have
an action of the group $\sdiff$ on $\cal D{}_{\omega}$, and then
on $\sns$.
\end{rem}
We would like now to give a characterization of transversally Hamiltonian
vector fields in terms of its monomial expansion (see Remark \ref{rem: monomial exp}).
Consider a monomial vector field
\begin{eqnarray*}
X & = & x^{k_{0}}y_{1}^{k_{1}}y_{2}^{k_{2}}S\left(\mu\right)\,\,\,\,\,,
\end{eqnarray*}
with $\mu=\left(\mu_{0},\mu_{1},\mu_{2}\right)\in\ww C^{3}\backslash\acc 0$,
such that $\cal L_{X}\left(x\right)\in\ps x$. We necessarily have
either $\mu_{0}=0$ or $k_{0}\geq0$. Let us compute its Lie derivative
applied to $\omega$:
\begin{eqnarray*}
\cal L_{X}\left(\omega\right) & = & -\frac{\cal L_{X}\left(x\right)}{x^{2}}\mbox{d}y_{1}\wedge\mbox{d}y_{2}+\frac{1}{x}\mbox{d}\left(\cal L_{X}\left(y_{1}\right)\right)\wedge\mbox{d}y_{2}+\frac{1}{x}\mbox{d}y_{1}\wedge\mbox{d}\left(\cal L_{X}\left(y_{2}\right)\right)\\
 & = & -\mu_{0}x^{k_{0}-1}y_{1}^{k_{1}}y_{2}^{k_{2}}\mbox{d}y_{1}\wedge\mbox{d}y_{2}+\frac{\mu_{1}}{x}\mbox{d}\left(x^{k_{0}}y_{1}^{k_{1}+1}y_{2}^{k_{2}}\right)\wedge\mbox{d}y_{2}\\
 &  & +\frac{\mu_{2}}{x}\mbox{d}y_{1}\wedge\mbox{d}\left(x^{k_{0}}y_{1}^{k_{1}}y_{2}^{k_{2}+1}\right)\\
 & = & \left(\mu_{1}\left(k_{1}+1\right)+\mu_{2}\left(k_{2}+1\right)-\mu_{0}\right)x^{k_{0}-1}y_{1}^{k_{1}}y_{2}^{k_{2}}\mbox{d}y_{1}\wedge\mbox{d}y_{2}+\ps{\mbox{d}x}\quad.
\end{eqnarray*}
Moreover: 
\begin{eqnarray*}
\cal L_{X}\left(\mbox{d}x\right) & = & \mbox{d}\left(\cal L_{X}\left(x\right)\right)\\
 & = & \mbox{d}\left(\mu_{0}x^{k_{0}+1}y_{1}^{k_{1}}y_{2}^{k_{2}}\right)\\
 & = & \mu_{0}\left(\left(k_{0}+1\right)x^{k_{0}}y_{1}^{k_{1}}y_{2}^{k_{2}}\mbox{d}x+k_{1}x^{k_{0}+1}y_{1}^{k_{1}-1}y_{2}^{k_{2}}\mbox{d}y_{1}+k_{2}x^{k_{0}+1}y_{1}^{k_{1}}y_{2}^{k_{2}-1}\mbox{d}y_{2}\right)\quad.
\end{eqnarray*}
 Thus, we see that $X$ is transversally Hamiltonian if and only if
the following two conditions hold:
\begin{enumerate}
\item $\mu_{1}\left(k_{1}+1\right)+\mu_{2}\left(k_{2}+1\right)=\mu_{0}$ 
\item either $\mu_{0}=0$ or $k_{1}=k_{2}=0$. 
\end{enumerate}
So we have the following:
\begin{prop}
Let $X\in\fvf$ be a singular vector field and let 
\[
X=\sum_{\mathbf{k}\in\cal I}\mathbf{x}^{\mathbf{k}}S\left(\mu_{\mathbf{k}}\right)
\]
be its monomial expansion. $X$ is transversally Hamiltonian if and
only if for all $\mathbf{k}\in\cal I$, $\mathbf{x}^{\mathbf{k}}S\left(\mu_{\mathbf{k}}\right)$
is transversally Hamiltonian. \end{prop}
\begin{proof}
Clearly if $\mathbf{x}^{\mathbf{k}}S\left(\mu_{\mathbf{k}}\right)$
is transversally Hamiltonian for all $\mathbf{k}\in\cal I$, then
$X$ is transversally Hamiltonian is obvious, by convergence of the
above series in the Krull topology. Assume conversely that $X$ is
transversally Hamiltonian. First of all, notice that we necessarily
have, for all $\mathbf{k}\in\cal I$, $\cal L_{\mathbf{x}^{\mathbf{k}}S\left(\mu_{\mathbf{k}}\right)}\left(\mbox{d}x\right)\in\ps{\mbox{d}x}$.
Indeed, if it were not the case, consider $\mathbf{k}$ with $\abs{\mathbf{k}}$
minimum among the set of multi-index $\mathbf{l}$ satisfying 
\[
\cal L_{x^{l}S\left(\mu_{\mathbf{l}}\right)}\left(\mbox{d}x\right)\notin\ps{\mbox{d}x}
\]
 to obtain a contradiction, by looking at the terms of higher order.
Similarly, according to the computation above, for each $\mathbf{k}\in\cal I$:
\begin{eqnarray*}
\cal L_{\mathbf{x}^{\mathbf{k}}S\left(\mu_{\mathbf{k}}\right)}\left(\omega\right) & = & \left(\mu_{1}\left(k_{1}+1\right)+\mu_{2}\left(k_{2}+1\right)-\mu_{0}\right)x^{k_{0}-1}y_{1}^{k_{1}}y_{2}^{k_{2}}\mbox{d}y_{1}\wedge\mbox{d}y_{2}+\ps{\mbox{d}x}
\end{eqnarray*}
If one of the two conditions $1.$ or $2.$ above were not satisfied
by a couple $\left(\mathbf{k},\mu_{\mathbf{k}}\right)$ with \textbf{$\abs{\mathbf{k}}$
}minimal, then we could not have $x\cal L_{X}\left(\omega\right)\in\ps{\mbox{d}x}$
(just consider the terms of higher order).
\end{proof}

\section{Formal classification under fibered transformations}

\subsection{\label{sub: non degenerate invariant well defined}Invariance of
the residue by fibered conjugacy}

~

We start this section by proving that the non-degenerate condition
defined in the introduction only depends on the conjugacy class of
the vector field under the action of fibered diffeomorphisms. More
precisely, the following proposition states that the residue is an
invariant of a doubly-resonant saddle-node under the action of $\fdiff$.
\begin{prop}
Let $X,Y\in\sn$. If $X$ and $Y$ are $\fdiff-$conjugate, then $\res X=\res Y$.\end{prop}
\begin{proof}
Consider the system
\begin{equation}
x^{2}\ddd{\mathbf{y}}x=\mathbf{\alpha}\left(x\right)+\mathbf{A}\left(x\right)\mathbf{y}\left(x\right)+\mathbf{f}\left(x,\mathbf{y}\left(x\right)\right)\qquad,\label{eq: system doubly resonant saddle node-1-2}
\end{equation}
with $\mathbf{y}=\left(y_{1},y_{2}\right)$ and where the following
conditions hold:
\begin{itemize}
\item ${\displaystyle \alpha\left(x\right)=\left(\begin{array}{c}
\alpha_{1}\left(x\right)\\
\alpha_{2}\left(x\right)
\end{array}\right)},$ with ${\displaystyle \alpha_{1},\alpha_{2}\in\ps x^{2}\subset\form x}$
\item $\mathbf{A}\left(x\right)\in\mbox{Mat}_{2,2}\left(\form x\right)$
with $\mathbf{A}\left(0\right)=\tx{Diag}\left(-\lambda,\lambda\right)$,
$\lambda\in\ww C^{*}$
\item ${\displaystyle \mathbf{f}\left(x,\mathbf{y}\right)=\left(\begin{array}{c}
f_{1}\left(x,\mathbf{y}\right)\\
f_{2}\left(x,\mathbf{y}\right)
\end{array}\right)}$, with $f_{1},f_{2}\in\ps{y_{1},y_{2}}^{2}\subset\form{x,\mathbf{y}}$.
\end{itemize}
Perform the change of coordinates given by $\mathbf{y}=\beta\left(x\right)+\mathbf{P}\left(x\right)\mathbf{z}+\mathbf{h}\left(x,\mathbf{z}\right)$,
with ${\displaystyle \mathbf{z}}=\left(z_{1},z_{2}\right)$ and where:
\begin{itemize}
\item ${\displaystyle \beta\left(x\right)=\left(\begin{array}{c}
\beta_{1}\left(x\right)\\
\beta_{2}\left(x\right)
\end{array}\right)},$ with ${\displaystyle \beta_{1},\beta_{2}\in\ps x\subset\form x}$
\item $\mathbf{P}\left(x\right)\in\mbox{Mat}_{2,2}\left(\form x\right)$
such that $\mathbf{P}\left(0\right)\in\tx{GL}_{2}\left(\ww C\right)$
\item ${\displaystyle \mathbf{h}\left(x,\mathbf{y}\right)=\left(\begin{array}{c}
h_{1}\left(x,\mathbf{y}\right)\\
h_{2}\left(x,\mathbf{y}\right)
\end{array}\right)}$, with $h_{1},h_{2}\in\ps{z_{1},z_{2}}^{2}\subset\form{x,\mathbf{z}}$.
\end{itemize}
Then one obtain the following system satisfied by $\mathbf{z}\left(x\right)$:
\begin{eqnarray*}
x^{2}\ddd{\mathbf{z}}x & = & \mathbf{P}\left(x\right)^{-1}\left(\mathbf{\alpha}\left(x\right)+\mathbf{A}\left(x\right)\beta\left(x\right)+\mathbf{f}\left(x,\beta\left(x\right)\right)-x^{2}\ddd{\beta}x\left(x\right)\right)\\
 &  & +\mathbf{P}\left(x\right)^{-1}\left(\mathbf{A}\left(x\right)\mathbf{P}\left(x\right)-x^{2}\ddd{\mathbf{P}}x\left(x\right)+\ppp{\mathbf{f}}{\mathbf{y}}\left(x,\beta\left(x\right)\right)\mathbf{P}\left(x\right)\right)\mathbf{z}+\mbox{\ensuremath{\ps{z_{1},z_{2}}^{2}}}\qquad,
\end{eqnarray*}
Since $\mathbf{A}\left(0\right)\in\tx{GL}_{2}\left(\ww C\right)$,
$\mathbf{f}\left(x,\mathbf{y}\right)\in\ps{y_{1},y_{2}}^{2}$ and
$\tx{ord}\left(\beta\right)\geq1$, the order of 
\[
\mathbf{P}\left(x\right)^{-1}\left(\mathbf{\alpha}\left(x\right)+\mathbf{A}\left(x\right)\beta\left(x\right)+\mathbf{f}\left(x,\beta\left(x\right)\right)-x^{2}\ddd{\beta}x\left(x\right)\right)
\]
 is at least 2 if and only if $\tx{ord}\left(\beta\right)\geq2$.
Then: 
\begin{eqnarray*}
\tx{Tr}\left(\mathbf{P}\left(x\right)^{-1}\left(\mathbf{A}\left(x\right)\mathbf{P}\left(x\right)-x^{2}\ddd{\mathbf{P}}x\left(x\right)+\ppp{\mathbf{f}}{\mathbf{y}}\left(x,\beta\left(x\right)\right)\mathbf{P}\left(x\right)\right)\right) & \in & \tx{Tr}\left(\mathbf{A}\left(x\right)\right)+\ps x^{2}.
\end{eqnarray*}
So ${\displaystyle \left(\frac{\mbox{Tr}\left(\mathbf{A}\left(x\right)\right)}{x}\right)_{\mid x=0}}$
is invariant by fibered change of coordinates on system of the form
$\left(\mbox{\ref{eq: system doubly resonant saddle node-1-2}}\right)$
with $\tx{ord}\left(\alpha\right)\geq2$.
\end{proof}

\subsection{Proof of Theorems \ref{Th: Th drsn} and \ref{thm: Th ham}}

~

We will use the tools described in Section $2.$ 
\begin{proof}
Let $Y\in\snnd$ (\emph{resp. }in\emph{ $\sns$}) be a non-degenerate
(\emph{resp.} transversally Hamiltonian) doubly-resonant saddle-node:
\begin{eqnarray*}
Y & = & x^{2}\pp x+\left(-\lambda y_{1}+F_{1}\left(x,y_{1},y_{2}\right)\right)\pp{y_{1}}+\left(\lambda y_{2}+F_{2}\left(x,y_{1},y_{2}\right)\right)\pp{y_{1}}\,\,\,\,,
\end{eqnarray*}
with $\lambda\in\ww C^{*}$, and $F_{\nu}\left(x,\mathbf{y}\right)\in\mathfrak{m}^{2}$,
for $\nu=1,2$ . As seen in the previous subsection, we can assume
that $F_{1}\left(x,0,0\right)=F_{2}\left(x,0,0\right)=0$.

The general idea is to apply successive (infinitely many) diffeomorphisms
of the form 
\[
\exp\left(x^{j_{0}}y_{1}^{j_{1}}y_{2}^{j_{2}}S\left(0,\mu_{1,\mathbf{j}},\mu_{2,\mathbf{j}}\right)\right)
\]
 for convenient choices of $\mathbf{j},\mu_{1,\mathbf{j}},\mu_{2,\mathbf{j}}$,
in order to remove all the terms we want to. Let us consider the monomial
expansion of $Y$: 
\begin{eqnarray}
Y & = & \lambda S\left(0,-1,1\right)+xS\left(1,0,0\right)+\sum_{\mathbf{k}\in\cal I,\,\abs{\mathbf{k}}\geq1}x^{k_{0}}y_{1}^{k_{1}}y_{2}^{k_{2}}S\left(0,\mu_{1,\mathbf{k}},\mu_{2,\mathbf{k}}\right)\quad.\label{eq: decomposition proof}
\end{eqnarray}
Since $Y$ in non-degenerate we necessarily have 
\[
\mu_{1,\left(1,00\right)}+\mu_{2,\left(1,0,0\right)}=\res Y\in\ww C\backslash\ww Q_{\leq0}\qquad.
\]
In the transversally Hamiltonian case, each term in the sum 
\[
\sum_{\mathbf{k}\in\cal I,\,\abs{\mathbf{k}}\geq1}x^{k_{0}}y_{1}^{k_{1}}y_{2}^{k_{2}}S\left(0,\mu_{1,\mathbf{k}},\mu_{2,\mathbf{k}}\right)
\]
must satisfy 
\[
\mu_{1,\mathbf{k}}\left(k_{1}+1\right)+\mu_{2,\mathbf{k}}\left(k_{2}+1\right)=0\qquad,
\]
if $\mathbf{k}\neq\left(1,0,0\right)$ and $\mu_{1,\left(1,0,0\right)}+\mu_{2,\left(1,0,0\right)}=1$.

The normalizing conjugacy $\Phi$ is constructed in two steps. 
\begin{enumerate}
\item The first step is aimed at removing all \textbf{non-resonant} monomial
terms, \emph{i.e. }those of the form 
\[
{\displaystyle x^{k_{0}}y_{1}^{k_{1}}y_{2}^{k_{2}}S\left(0,\mu_{1,\mathbf{k}},\mu_{2,\mathbf{k}}\right)}\mbox{ , with }\mathbf{k}\in\cal I,\abs{\mathbf{k}}\geq1\mbox{ and }k_{1}\neq k_{2}\qquad.
\]
 
\item The second step is aimed at removing certain resonant monomial terms,
and more precisely those of the form 
\[
{\displaystyle x^{k_{0}}\left(y_{1}y_{2}\right)^{k}S\left(0,\eta_{1,\mathbf{i}},\eta_{2,\mathbf{i}}\right)}\mbox{ , except for}\left(k_{0},k\right)=\left(1,0\right)\mbox{ and }k_{0}=0\qquad.
\]

\end{enumerate}
We will see that each one of these steps allows us to define a fibered
diffeomorphism $\Phi_{j}$ (transversally symplectic in the transversally
Hamiltonian case), for $j=1,2$. Finally we define $\Phi:=\Phi_{2}\circ\Phi_{1}$.
The main tool used at each step is Proposition \ref{prop: exp }.
Moreover, each $\Phi_{j}$ will be constructed using Corollary \ref{lem: diffeo formel infinite composition}.
The fact that each $\Phi_{j}$ is a fibered diffeomorphism (transversally
symplectic in the transversally Hamiltonian case) will again come
from Lemma \ref{lem: diffeo formel infinite composition} (and Lemma
\ref{lem: suite de diffeo symplectique} in the transversally symplectic
case, and each $Y_{j}=\left(\Phi_{j}\right)_{*}\left(Y_{j-1}\right)$,
$j=1,2$ with $Y_{0}:=Y$, will be transversally Hamiltonian according
to Lemma \ref{lem: push forward symplectic fibr=0000E9}).
\begin{enumerate}
\item \emph{First step:} we remove all non-resonant monomial terms, using
diffeomorphisms of the form 
\[
{\displaystyle \exp\left(x^{i_{0}}y_{1}^{i_{1}}y_{2}^{i_{2}}S\left(0,\eta_{1,\mathbf{i}},\eta_{2,\mathbf{i}}\right)\right)}\quad,
\]
with $\mathbf{i}\in\cal I$, $\abs{\mathbf{i}}\geq1$, $i_{1}\neq i_{2}$
and $\eta_{1,\mathbf{i}},\eta_{2,\mathbf{i}}$ to be determined. We
have, thanks to Proposition \ref{prop: exp }: 
\[
{\displaystyle \Bigg(\exp\left(x^{i_{0}}y_{1}^{i_{1}}y_{2}^{i_{2}}S\left(0,\eta_{1,\mathbf{i}},\eta_{2,\mathbf{i}}\right)\right)}\Bigg)_{*}\left(Y_{0}\right)=Y_{0}+\frac{1}{1!}\cro{x^{i_{0}}y_{1}^{i_{1}}y_{2}^{i_{2}}S\left(0,\eta_{1,\mathbf{i}},\eta_{2,\mathbf{i}}\right),Y_{0}}+\dots,
\]
where $\left(\dots\right)$ are terms computed \emph{via} successive
nested brackets, and they are all of order at least $\abs{\mathbf{i}}+1$.
Let us compute the first bracket:{\small{}
\begin{eqnarray*}
 &  & \cro{x^{i_{0}}y_{1}^{i_{1}}y_{2}^{i_{2}}S\left(0,\eta_{1,\mathbf{i}},\eta_{2,\mathbf{i}}\right),Y_{0}}\\
 & = & \lambda\left(i_{1}-i_{2}\right)x^{i_{0}}y_{1}^{i_{1}}y_{2}^{i_{2}}S\left(0,\eta_{1,\mathbf{i}},\eta_{2,\mathbf{i}}\right)\\
 &  & -i_{0}x^{i_{0}+1}y_{1}^{i_{1}}y_{2}^{i_{2}}S\left(0,\eta_{1,\mathbf{i}},\eta_{2,\mathbf{i}}\right)\\
 &  & +\sum_{\mathbf{k}\in\cal I,\,\substack{\abs{\mathbf{k}}\geq1}
}x^{i_{0}+k_{0}}y_{1}^{i_{1}+k_{1}}y_{2}^{i_{2}+k_{2}}\left(k_{1}\eta_{1,i_{0}}+k_{2}\eta_{2,i_{2}}\right)S\left(0,\mu_{1,\mathbf{k}},\mu_{2,\mathbf{k}}\right)\\
 &  & -\sum_{\substack{\mathbf{k}\in\cal I,\,\abs{\mathbf{k}}\geq1}
}x^{i_{0}+k_{0}}y_{1}^{i_{1}+k_{1}}y_{2}^{i_{2}+k_{2}}\left(i_{1}\mu_{1,\mathbf{k}}+i_{2}\mu_{2,\mathbf{k}}\right)S\left(0,\eta_{1,\mathbf{i}},\eta_{2,\mathbf{i}}\right)\quad.
\end{eqnarray*}
} Then one can remove all terms of the form ${\displaystyle x^{i_{0}}y_{1}^{i_{1}}y_{2}^{i_{2}}S\left(0,\mu_{1,\mathbf{i}},\mu_{2,\mathbf{i}}\right)}$
with $\abs{\mathbf{i}}\geq1$ and $i_{1}\neq i_{2}$ by induction
on $\abs{\mathbf{i}}\geq1$. We then define (using Lemma \ref{lem: diffeo formel infinite composition})
a fibered diffeomorphism $\Phi_{1}$, such that $Y_{1}:=\left(\Phi_{1}\right)_{*}\left(Y_{0}\right)$
is still of the form $\left(\ref{eq: decomposition proof}\right)$,
but without non-resonant terms: 
\begin{eqnarray*}
Y_{1} & = & \lambda S\left(0,-1,1\right)+xS\left(1,a_{1},a_{2}\right)\\
 &  & +\sum_{\substack{k_{0}+k\geq1\\
\left(k_{0},k\right)\neq\left(1,0\right)
}
}x^{k_{0}}y_{1}^{k}y_{2}^{k}S\left(0,\mu_{1,\mathbf{k}},\mu_{2,\mathbf{k}}\right)
\end{eqnarray*}
for maybe different $\mu_{j,\mathbf{k}}$. Notice that $a_{1},a_{2}$
here are necessarily such that $a_{1}+a_{2}\notin\ww Q_{\leq0}$ since
the vector field is supposed to be non-degenerate, and this condition
is invariant under fibered change of coordinates. 

\begin{rem*}
In the transversally Hamiltonian case, the terms $x^{i_{0}}y_{1}^{i_{1}}y_{2}^{i_{2}}S\left(0,\eta_{1,\mathbf{i}},\eta_{2,\mathbf{i}}\right)$
to be removed at this stage satisfy $\eta_{1,\mathbf{i}}\left(i_{1}+1\right)+\eta_{2,\mathbf{i}}\left(i_{2}+1\right)=0$,
so that $\Phi_{1}$ is transversally symplectic according to Proposition
\ref{prop: expo de champ transv hamilt} and Lemma \ref{lem: suite de diffeo symplectique}.
Moreover, in this case, we necessarily have $a_{1}+a_{2}=1$.
\end{rem*}
\item \emph{Second step:} we finally remove all the terms of the form 
\[
{\displaystyle x^{i_{0}}\left(y_{1}y_{2}\right)^{i}S\left(0,\eta_{1,\mathbf{i}},\eta_{2,\mathbf{i}}\right)}\mbox{ , except for}\left(i_{0},i\right)=\left(1,0\right)\mbox{ and }i_{0}=0\qquad,
\]
using diffeomorphisms of the form 
\[
{\displaystyle \exp\left(x^{i_{0}}\left(y_{1}y_{2}\right)^{i}S\left(0,\eta_{1,\mathbf{i}},\eta_{2,\mathbf{i}}\right)\right)}\qquad,
\]
with $i_{0}+i\geq1$, and $\eta_{1,\mathbf{i}},\eta_{2,\mathbf{i}}$
to be determined. We have, thanks to Proposition \ref{prop: exp }:
\[
{\displaystyle \Bigg(\exp\left(x^{i_{0}}\left(y_{1}y_{2}\right)^{i}S\left(0,\eta_{1,\mathbf{i}},\eta_{2,\mathbf{i}}\right)\right)}\Bigg)_{*}\left(Y_{1}\right)=Y_{1}+\frac{1}{1!}\cro{x^{i_{0}}\left(y_{1}y_{2}\right)^{i}S\left(0,\eta_{1,\mathbf{i}},\eta_{2,\mathbf{i}}\right),Y_{1}}+\dots,
\]
where $\left(\dots\right)$ are terms computed \emph{via} successive
nested brackets, and they are all of order strictly greater than the
order of the first bracket. Let us compute the first bracket:{\small{}
\begin{eqnarray*}
 &  & \cro{x^{i_{0}}\left(y_{1}y_{2}\right)^{i}S\left(0,\eta_{1,\mathbf{i}},\eta_{2,\mathbf{i}}\right),Y_{3}}\\
 & = & -\left(i_{0}+i\left(a_{1}+a_{2}\right)\right)x^{i_{0}+1}\left(y_{1}y_{2}\right)^{i}S\left(0,\eta_{1,\mathbf{i}},\eta_{2,\mathbf{i}}\right)\\
 &  & +{\displaystyle \sum_{\substack{k_{0}+2k\geq1\\
\left(k_{0},k\right)\neq\left(1,0\right)
}
}}x^{i_{0}+k_{0}}\left(y_{1}y_{2}\right)^{i+k}k\left(\eta_{1,i_{0}}+\eta_{2,i_{2}}\right)S\left(0,\mu_{1,\mathbf{k}},\mu_{2,\mathbf{k}}\right)\\
 &  & {\displaystyle -\sum_{\substack{k_{0}+2k\geq1\\
\left(k_{0},k\right)\neq\left(1,0\right)
}
}x^{i_{0}+k_{0}}\left(y_{1}y_{2}\right)^{i+k}i\left(\mu_{1,\mathbf{k}}+\mu_{2,\mathbf{k}}\right)S\left(0,\eta_{1,\mathbf{i}},\eta_{2,\mathbf{i}}\right)}\,\,.
\end{eqnarray*}
} Then we see that one can remove all terms of the form ${\displaystyle x^{i_{0}}\left(y_{1}y_{2}\right)^{i}S\left(0,\eta_{1,\mathbf{i}},\eta_{2,\mathbf{i}}\right)}$
except for $\left(i_{0},i\right)=\left(1,0\right)$ and for $i_{0}=0$,
without creating non-resonant terms, since $\left(a_{1}+a_{2}\right)\notin\ww Q_{\leq0}$
. We do this by induction on $I:=i_{0}+i\geq1$, and for fixed $I\geq1$,
we remove the terms with $i$ increasing and $i_{0}$ decreasing.
Notice that at each step we do not create terms already removed earlier
in the process.\\
We then define (using Lemma \ref{lem: diffeo formel infinite composition})
a fibered diffeomorphism $\Phi_{2}$, such that $Y_{2}:=\left(\Phi_{2}\right)_{*}\left(Y_{1}\right)$
is of the form 
\begin{eqnarray*}
Y_{2} & = & \lambda S\left(0,-1,1\right)+xS\left(1,a_{1},a_{2}\right)+\sum_{k\geq1}\left(y_{1}y_{2}\right)^{k}S\left(0,\mu_{1,\mathbf{k}},\mu_{2,\mathbf{k}}\right)\quad.
\end{eqnarray*}

\begin{rem*}
In the transversally Hamiltonian case, the terms $x^{i_{0}}\left(y_{1}y_{2}\right)^{i}S\left(0,\eta_{1,\mathbf{i}},\eta_{2,\mathbf{i}}\right)$
to be removed at this stage satisfy $\left(\eta_{1,\mathbf{i}}+\eta_{2,\mathbf{i}}\right)=0$,
so that $\Phi_{2}$ is transversally symplectic, according to Proposition
\ref{prop: expo de champ transv hamilt} and and Lemma \ref{lem: suite de diffeo symplectique}.
\end{rem*}
\end{enumerate}
Finally, we define $\Phi:=\Phi_{2}\circ\Phi_{1}$, so that $\Phi_{*}\left(Y\right)=Y_{2}$
and $\Phi$ is a fibered diffeomorphism (transversally symplectic
in the Hamiltonian case).
\end{proof}

\subsection{\label{sub: proof of 4. in Th1}Uniqueness: proof of Theorem \ref{thm: uniqueness}}

~

We now prove Theorem \ref{thm: uniqueness}.
\begin{proof}
Let
\begin{eqnarray*}
Z & = & x^{2}\pp x+\left(-\lambda+a_{1}x+c_{1}\left(v\right)\right)z_{1}\pp{z_{1}}+\left(\lambda+a_{2}x+c_{2}\left(v\right)\right)z_{2}\pp{z_{2}}\\
Z' & = & x^{2}\pp x+\left(-\lambda'+a'_{1}x+c'_{1}\left(v\right)\right)z_{1}\pp{z_{1}}+\left(\lambda'+a'_{2}x+c'_{2}\left(v\right)\right)z_{2}\pp{z_{2}}\,\,\,\,\,,
\end{eqnarray*}
where $\left(\lambda,\lambda',a_{1},a_{2},a'_{1},a'_{2}\right)\in\left(\ww C^{*}\right)^{2}\times\ww C^{4}$,
$\left(a_{1}+a_{2},a'_{1}+a'_{2}\right)\in\left(\ww C\backslash\ww Q_{\leq0}\right)^{2}$
and $\left(c_{1},c_{2},c'_{1},c'_{2}\right)\in\left(v\form v\right)^{4}$
are formal power series in $v=z_{1}z_{2}$ of order at least one.
\begin{itemize}
\item It is clear that if there exists $\varphi\,:\, v\mapsto\theta v$
with $\theta\in\ww C^{*}$ such that 
\begin{eqnarray*}
\left(\lambda,a_{1},a_{2},c_{1},c_{2}\right) & = & \left(\lambda',a'_{1},a'_{2},c'_{1}\circ\varphi,c'_{2}\circ\varphi\right)\\
\Bigg(\mbox{\emph{resp.} }\,\left(\lambda,a_{1},a_{2},c_{1},c_{2}\right) & = & \left(-\lambda',a'_{2},a'_{1},c'_{2}\circ\varphi,c'_{1}\circ\varphi\right)\Bigg)
\end{eqnarray*}
 then $Z$ is $\fdiff$-conjugate to $Z'$.
\item Now assume that $Z$ is $\fdiff$-conjugate to $Z'$. First of all,
studying the terms of degree 1 with respect to $\mathbf{z}$, we see
that we either have $\left(\lambda,a_{1},a_{2}\right)=\left(\lambda',a'_{1},a'_{2}\right)$
or $\left(\lambda,a_{1},a_{2}\right)=\left(-\lambda',a'_{2},a'_{1}\right)$.
Up to perform a linear change of coordinates beforehand, let us assume
that $\left(\lambda,a_{1},a_{2}\right)=\left(\lambda',a'_{1},a'_{2}\right)$.
In the following, and for convenience, we will use the notations:
\[
\begin{cases}
Z=Z_{\left(c,r\right)}:=xS\left(1,a_{1},a_{2}\right)+\left(\lambda+c\left(v\right)\right)S\left(0,-1,1\right)+r\left(v\right)S\left(0,a_{1},a_{2}\right)\\
Z'=Z_{\left(c',r'\right)}:=xS\left(1,a_{1},a_{2}\right)+\left(\lambda+c'\left(v\right)\right)S\left(0,-1,1\right)+r'\left(v\right)S\left(0,a_{1},a_{2}\right) & ,
\end{cases}
\]
where: 
\begin{eqnarray*}
\begin{cases}
c_{1}=-c+r & ,\, c_{2}=c+r\\
c'_{1}=-c'+r' & ,\, c'_{2}=c'+r'
\end{cases}\,\,\,\,,
\end{eqnarray*}
so that $\mbox{ord}\left(c\right)\geq1$, $\mbox{ord}\left(r\right)\geq1$.
\\
Now we have to prove that if $Z_{\left(c,r\right)}$ is $\fdiff-$conjugate
to $Z_{\left(c',r'\right)}$, then $\left(c,r\right)=\left(c',r'\right)$.
By assumption, there exists $\Phi\in\fdiff$ such that 
\begin{eqnarray*}
\Phi_{*}\left(Z_{\left(c,r\right)}\right) & = & Z_{\left(c',r'\right)}\,\,\,.
\end{eqnarray*}
By Remark \ref{rem: D is diagonal}, $\mbox{D}_{0}\Phi=\tx{diag}\left(1,\theta_{1},\theta_{2}\right)$
is diagonal. Now, set $\Psi:=\left(\mbox{D}_{0}\Phi\right)^{-1}\circ\Phi$
, $\varphi:v\mapsto\left(\theta_{1}\theta_{2}\right)v$, and $\left(\overline{c},\overline{r}\right):=\left(c'\circ\varphi,r'\circ\varphi\right)$,
so that: 
\begin{eqnarray*}
\Psi_{*}\left(Z_{\left(c,r\right)}\right) & = & Z_{\left(\overline{c},\overline{r}\right)}\,\,\,\,.
\end{eqnarray*}
We are going to prove that $\Psi=\mbox{Id}$. By Proposition \ref{prop: conjuation with resonant monomial},
there exists $G\in\fvf$ such that $\Psi=\tx{exp}\left(G\right)$
and: 
\begin{eqnarray*}
G & = & g_{0}\left(x,v\right)\pp x+g_{1}\left(x,v\right)z_{1}\pp{z_{1}}+g_{2}\left(x,v\right)z_{2}\pp{z_{2}}\,\,\,\,,
\end{eqnarray*}
where $g_{i}\in\mathfrak{m}\subset\form{x,v}$ for $i=1,2$ and $g_{0}\in\mathfrak{m}^{2}\subset\form{x,v}$
is of order at least two. Since $\Psi$ is fibered in $x$ we deduce
that $g_{0}=0$. Therefore, using the notation $\left(\mbox{\ref{eq: log notation}}\right)$,
we can write: 
\begin{eqnarray*}
G & = & A\left(x,v\right)S\left(0,-1,1\right)+B\left(x,v\right)S\left(0,a_{1},a_{2}\right)\,\,\,\,,
\end{eqnarray*}
where 
\[
\begin{cases}
\math{A=A_{i,j}x^{i}v^{j}}\\
\math{B=\sum_{\substack{i,j\geq0\\
i+j\geq1
}
}B_{i,j}x^{i}v^{j}}
\end{cases}\,\,\,\,.
\]
Let us prove that $A=B=0$ (hence $G=0$) so that $\Psi=\mbox{Id}$.
We consider the Jordan decompositions of $Z:=Z_{\left(c,r\right)}$
and $\overline{Z}:=Z_{\left(\overline{c},\overline{r}\right)}$: 
\begin{eqnarray*}
\begin{cases}
Z=Z_{S}+Z_{N} & ,\, Z_{S}\mbox{ semi-simple},\, Z_{N}\mbox{ nilpotent},\,\cro{Z_{S},Z_{N}}=0\\
\overline{Z}=\overline{Z}{}_{S}+\overline{Z}{}_{N} & ,\,\overline{Z}{}_{S}\mbox{ semi-simple},\,\overline{Z}{}_{N}\mbox{ nilpotent},\,\cro{\overline{Z}{}_{S},\overline{Z}{}_{N}}=0
\end{cases}\,\,\,\,.
\end{eqnarray*}
By uniqueness of this decomposition we clearly have:
\begin{eqnarray*}
\begin{cases}
\math{Z_{S}=\overline{Z}{}_{S}=S\left(0,-\lambda,\lambda\right)}\\
\math{Z_{N}=xS\left(1,a_{1},a_{2}\right)+c\left(v\right)S\left(0,-1,1\right)+r\left(v\right)S\left(0,a_{1},a_{2}\right)}\\
\math{\overline{Z}_{N}=xS\left(1,a_{1},a_{2}\right)+\overline{c}\left(v\right)S\left(0,-1,1\right)+\overline{r}\left(v\right)S\left(0,a_{1},a_{2}\right)}
\end{cases}\,\,\,\,,
\end{eqnarray*}
and we also know that: 
\begin{eqnarray*}
\Psi_{*}\left(Z\right)=\overline{Z} & \Rightarrow & \begin{cases}
\Psi_{*}\left(Z_{S}\right)=\overline{Z}{}_{S}\\
\Psi_{*}\left(Z_{N}\right)=\overline{Z}{}_{N}
\end{cases}\,\,\,\,.
\end{eqnarray*}
Let us now consider the associated two-dimensional vector fields in
the variables $\left(x,v\right)$. In the ``chart'' $\left(x,v\right)$
the vector field $G$ is given by $F=B.S\left(0,a\right)$, with $a=a_{1}+a_{2}$.
$Z$ and $\overline{Z}$ correspond respectively to: 
\begin{eqnarray*}
Y & := & xS\left(1,a\right)+r\left(v\right)S\left(0,a\right)\,\,\,,\\
\overline{Y} & := & xS\left(1,a\right)+\overline{r}\left(v\right)S\left(0,a\right)\,\,\,.
\end{eqnarray*}
Thus we have $\mbox{exp}\left(F\right)_{*}\left(Y\right)=\overline{Y}$.
By Proposition \ref{prop: exp (ad)} we derive 
\begin{eqnarray*}
\tx{exp}\left(F\right)_{*}\left(Y\right) & = & Y+\cro{F,Y}+\frac{1}{2!}\cro{F,\cro{F,Y}}+\dots
\end{eqnarray*}
and 
\begin{eqnarray}
r\left(v\right)S\left(0,a\right)+\cro{F,Y}+\frac{1}{2!}\cro{F,\cro{F,Y}}+\dots & = & \overline{r}\left(v\right)S\left(0,a\right)\,\,\,\,.\label{eq: r and bar r}
\end{eqnarray}
We compute next 
\begin{eqnarray*}
\cro{F,Y} & = & \acc{-x\left(\cal L_{S\left(1,a\right)}\left(B\right)\right)+B\left(\cal L_{S\left(0,a\right)}\left(r\right)\right)-r\left(\cal L_{S\left(0,a\right)}\left(B\right)\right)}S\left(0,a\right)
\end{eqnarray*}
and, setting 
\begin{eqnarray}
C^{\left(1\right)}\left(x,v\right) & := & -x\left(\cal L_{S\left(1,a\right)}\left(B\right)\right)+B\left(\cal L_{S\left(0,a\right)}\left(r\right)\right)\label{eq: expression of C^(1)}\\
 &  & -r\left(\cal L_{S\left(0,a\right)}\left(B\right)\right)\,\,\,,\nonumber 
\end{eqnarray}
we obtain 
\begin{eqnarray*}
\cro{F,Y} & = & C^{\left(1\right)}\left(x,v\right)S\left(0,a\right)\,\,\,\,.
\end{eqnarray*}
Now, it is easy to see that for all $l\in\ww N$, $\mbox{ad}_{F}^{\circ l}\left(Y\right)$
ca be written 
\begin{eqnarray*}
\mbox{ad}_{F}^{\circ l}\left(Y\right) & = & C^{\left(l\right)}\left(x,v\right)S\left(0,a\right)\,\,\,\,,
\end{eqnarray*}
where $C^{\left(l\right)}$ is determined by the recursive relation:
\begin{eqnarray*}
C^{\left(l+1\right)}\left(x,v\right) & = & B\left(x,v\right)\left(\cal L_{S\left(0,a\right)}\left(C^{\left(l\right)}\right)\right)-C^{\left(l\right)}\left(x,v\right)\left(\cal L_{S\left(0,a\right)}\left(B\right)\right)\,\,\,\,.
\end{eqnarray*}
In particular, we see that for all $l\geq2$, $C^{\left(l\right)}\left(x,0\right)=0$.
Equation $\left(\mbox{\ref{eq: r and bar r}}\right)$ can now be rewritten:
\begin{eqnarray}
r\left(v\right)+C^{\left(1\right)}\left(x,v\right)+\underset{l\geq2}{\sum}C^{\left(l\right)}\left(x,v\right) & = & \overline{r}\left(v\right)\,\,\,\,.\label{eq: r, bar r and the C^(l)'s}
\end{eqnarray}
Let us set $\math{r\left(v\right)=\underset{k\geq1}{\sum}r_{k}v^{k}}$
and $\math{\overline{r}\left(v\right)=\underset{k\geq1}{\sum}\overline{r}_{k}v^{k}}$.
Looking at terms independent of $v$ in $\left(\mbox{\ref{eq: r, bar r and the C^(l)'s}}\right)$
($i.e.$ by taking $v=0$), we see that $C^{\left(1\right)}\left(x,0\right)=0$.
Taking $\left(\mbox{\ref{eq: expression of C^(1)}}\right)$ into account
we obtain that $\math{\ppp{B\left(x,0\right)}x=0}$. Since $\mbox{ord}\left(B\right)\geq1$
(by assumption) this means that $B\left(x,0\right)=0$. Let us prove
the properties $B_{i,k}=0$ and $r_{k}=\overline{r}_{k}$ for all
$i,j\in\ww N$ and $k\leq j$ by induction on $j\geq0$. 

\begin{itemize}
\item $j=0$. This corresponds to the case described above: for all $i\geq0$,
$B_{i,0}=0$ (and $r_{0}=\overline{r}_{0}$).
\item If the property holds at a rank $j\geq0$, if we consider for all
$i\geq0$ terms of homogenous degree $\left(i+1,j+1\right)$ in $\left(\mbox{\ref{eq: r, bar r and the C^(l)'s}}\right)$,
we obtain: 
\[
\left(i+a\left(j+1\right)\right)B_{i,j+1}=0
\]
by induction, and because for all $l\geq2$ the relation $C^{\left(l\right)}\left(x,0\right)=0$
also holds. Since $a\notin\ww Q_{\leq0}$ we have $B_{i,j+1}=0$.
On the other hand, if we look at terms of homogeneous degree $\left(0,j+1\right)$,
we obtain: $r_{j+1}=\overline{r}_{j+1}$. 
\end{itemize}
\end{itemize}

We conclude that $B=0$, so that $F=0$ and $r=\overline{r}$.

Finally, we have $G=A\cal L\left(0,-1,1\right)$. Taking the relation
$\tx{exp}\left(G\right)_{*}\left(Z_{N}\right)=\overline{Z}{}_{N}$
into account, we have:
\begin{eqnarray*}
Z_{N}+\cro{G,Z_{N}}+\frac{1}{2!}\cro{G,\cro{G,Z_{N}}}+\dots & = & \overline{Z}{}_{N}
\end{eqnarray*}
if and only if 
\begin{eqnarray*}
c\left(v\right)S\left(0,-1,1\right)+\cro{G,Z_{N}}+\frac{1}{2!}\cro{G,\cro{G,Z_{N}}}+\dots & = & \overline{c}\left(v\right)S\left(0,-1,1\right)\,\,\,\,\,.
\end{eqnarray*}
Let us compute $\cro{G,Z_{N}}$: 
\begin{eqnarray*}
\cro{G,Z_{N}} & = & -\acc{x\cal L_{S\left(1,a_{1},a_{2}\right)}\left(A\right)+r\left(v\right)\cal L_{S\left(0,a_{1},a_{2}\right)}\left(A\right)}S\left(0,-1,1\right)\,\,\,\,.
\end{eqnarray*}
All other Lie brackets vanish. There only remains
\begin{eqnarray*}
c\left(v\right)-x\cal L_{S\left(1,a_{1},a_{2}\right)}\left(A\right)-r\left(v\right)\cal L_{S\left(0,a_{1},a_{2}\right)}\left(A\right) & = & \overline{c}\left(v\right)\,\,\,\,,
\end{eqnarray*}
which becomes a system of identities between terms of same degree:
\begin{eqnarray*}
\begin{cases}
\math{c_{j}-\underset{k=0}{\overset{j-k_{1}}{\sum}}akA_{0,k}r_{j-k}=\overline{c}{}_{j}} & ,\, j\geq0\\
\math{\left(i+aj\right)A_{i,j}+\underset{k=0}{\overset{j-k_{1}}{\sum}}akA_{i+1,k}r_{j-k}=0} & ,\, i\geq0,\, j\geq0
\end{cases} &  & \,\,\,\,\,.
\end{eqnarray*}
Once again, we prove by induction on $j\geq0$ that for all $i\geq0$
and all $0\leq k\leq j$ the relations $A_{i,k}=0$ and $c_{k}=\overline{c}_{k}$
hold. Thus $A=0$ and $c=\overline{c}$.

As a conclusion $\Psi=\mbox{Id}$ and $\left(c,r\right)=\left(\overline{c},\overline{r}\right)$,
so that $\Phi=\mbox{D}_{0}\Phi=\tx{diag}\left(1,\theta_{1},\theta_{2}\right)$
and $\left(c,r\right):=\left(c'\circ\varphi,r'\circ\varphi\right)$
where $\varphi~:~v\mapsto\left(\theta_{1}\theta_{2}\right)v$.

\end{proof}

\subsection{Fibered isotropies of the formal normal form}

Looking back at the uniqueness proof in the previous paragraph, we
immediately obtain all formal fibered isotropies of the normal form
given by Theorems \ref{Th: Th drsn} and \ref{thm: Th ham}. We recall
that an isotropy of a vector field is a self-conjugacy. For a vector
field $X\in\fvf$, we set: 
\begin{eqnarray*}
\fisot X & := & \acc{\Phi\in\fdiff\mid\Phi_{*}\left(X\right)=X}\,\,\,\,.
\end{eqnarray*}

\begin{prop}
\label{prop: fibered isotropies}Consider a normal form of $\snnd$
\begin{eqnarray*}
Z & = & x^{2}\pp x+\left(-\lambda+a_{1}x+c_{1}\left(y_{1}y_{2}\right)\right)y_{1}\pp{y_{1}}+\left(\lambda+a_{2}x+c_{2}\left(y_{1}y_{2}\right)\right)y_{2}\pp{y_{2}}\,\,,
\end{eqnarray*}
with parameters $\left(\lambda,a_{1},a_{2},c_{1},c_{2}\right)\in\cal P$.
Then:
\begin{eqnarray*}
\fisot Z & = & \acc{\tx{diag}\left(1,\theta_{1},\theta_{2}\right),\,\left(\theta_{1},\theta_{2}\right)\in\left(\ww C^{*}\right)^{2}\Big\vert\left(c_{1},c_{2}\right)\left(\theta_{1}\theta_{2}v\right)=\left(c_{1},c_{2}\right)\left(v\right)}\,\,\,.
\end{eqnarray*}
\end{prop}
\begin{rem}
If $\left(c_{1},c_{2}\right)\neq\left(0,0\right)$ the condition $c_{i}\left(\theta_{1}\theta_{2}v\right)=c_{i}\left(v\right)$
for each $i\in\acc{1,2}$ is equivalent to requiring that each $c_{i}$
lie in $\form{v^{q}}$, for some $q\in\ww N_{>0}$, and that $\theta_{1}\theta_{2}$
be a $q^{\mbox{th}}$ root of unity.
\end{rem}
This proposition has for immediate consequence the (almost) uniqueness
of the normalizing conjugacy $\Phi\in\diff{}$ in Theorem \ref{Th: Th drsn}.
More precisely:
\begin{cor}
\label{cor: uniqueness tangent to id}Let $Y\in\snnd$ be a non-degenerate
doubly-resonant saddle-node such that $\tx D_{0}Y=\tx{diag}\left(0,-\lambda,\lambda\right)$,
with $\lambda\neq0$. Then there exists a unique fibered diffeomorphism
$\Phi\in\fdiff$ tangent to the identity such that: 
\begin{eqnarray}
\Phi_{*}\left(Y\right) & = & x^{2}\pp x+\left(-\lambda+a_{1}x+c_{1}\left(v\right)\right)y_{1}\pp{y_{1}}\nonumber \\
 &  & +\left(\lambda+a_{2}x+c_{2}\left(v\right)\right)y_{2}\pp{y_{2}}\,\,\,\,,\label{eq:eq: fibered normal form-1-2}
\end{eqnarray}
where we put $v:=y_{1}y_{2}$. Here, $c_{1},c_{2}$ belong to $\ps v=v\form v$
and $a_{1},a_{2}\in\ww C$ are such that $a_{1}+a_{2}=\res Y$.\end{cor}
\begin{defn}
Let $Z\in\sns$. We denote by $\widehat{\mbox{Isot}}_{\omega}\left(Z\right)$
the subgroup of elements $\Phi\in\sdiff$ such that $\Phi_{*}\left(Z\right)=Z$.\end{defn}
\begin{prop}
\label{prop: isot symp}Let $\left(\lambda,a_{1},a_{2}\right)\in\ww C^{*}\times\ww C^{2}$
such that $a_{1}+a_{2}=1$, and $c\in v\form v$ with $v=y_{1}y_{2}$.
Consider 
\begin{eqnarray*}
Z & = & x^{2}\pp x+\left(-\left(\lambda+c\left(v\right)\right)+a_{1}x\right)y_{1}\pp{y_{1}}+\left(\lambda+c\left(v\right)+a_{2}x\right)y_{2}\pp{y_{2}}\,\,\,.
\end{eqnarray*}
Then: 
\begin{eqnarray*}
\widehat{\mbox{\ensuremath{\tx{Isot}}}}_{\omega}\left(Z\right) & = & \acc{\tx{diag}\left(1,\alpha,\frac{1}{\alpha}\right),\,\alpha\in\ww C\backslash\acc 0}\simeq\ww C\backslash\acc 0.
\end{eqnarray*}

\end{prop}

\section{Applications to Painlevé equations}

In this section we investigate the study of the irregular singularity
at infinity in the first Painlevé equation 
\begin{eqnarray*}
\left(P_{I}\right)\,\,\,\,\,\,\,\,\,\,\,\,\,\,\,\,\,\,\,\,\,\ddd{^{2}z_{1}}{t^{2}} & = & 6z_{1}^{2}+t\,\,\,\,\,\,\,\,\,\,\,\,\,\,
\end{eqnarray*}
in terms of Theorem \ref{thm: Th ham}. More precisely, we are going
to explain that the formal invariant $c\in\form v$ of a doubly-resonant,
transversally symplectic saddle-node $Y\in\sns$ is in fact a germ
of an analytic function at the origin, whenever $Y$ is analytic at
the origin (and not merely a formal vector field). Moreover, we show
how to compute recursively this invariant in some specific cases,
including Painlevé equations.

\subsection{Asymptotically Hamiltonian vector fields}

We deal here with the case of \textbf{asymptotically Hamiltonian}
vector fields.
\begin{defn}
\label{def: asympt hamil}~
\begin{itemize}
\item We say that a formal vector field $X$ in $\left(\ww C^{2},0\right)$
is \textbf{orbitally linear} if 
\[
X=U\left(\mathbf{y}\right)\left(\lambda_{1}y_{1}\pp{y_{1}}+\lambda_{2}y_{2}\pp{y_{2}}\right)\,\,,
\]
for some unity ${\displaystyle U\left(\mathbf{y}\right)\in\form{\mathbf{y}}^{\times}}$
(\emph{i.e.} $U\left(0,0\right)\neq0$) and $\left(\lambda_{1},\lambda_{2}\right)\in\ww C^{2}$.
\item We say that a formal (\emph{resp. }germ of an analytic) vector field
$X$ in $\left(\ww C^{2},0\right)$ is formally (\emph{resp. }analytically\emph{)}
\textbf{orbitally linearizable} if $X$ is formally (\emph{resp.}
analytically) conjugate to an orbitally linear vector field.
\item We say that a doubly-resonant saddle-node $Y\in\sn$ is formally/analytically
\textbf{asymptotically orbitally linearizable} if the formal/analytic
vector field ${\displaystyle Y_{\mid\acc{x=0}}}$ in $\left(\ww C^{2},0\right)$
is formally/analytically orbitally linearizable. 
\end{itemize}
\end{defn}
\begin{rem}
~
\begin{enumerate}
\item If a vector field $X$ is analytic at the origin of $\ww C^{2}$ and
has two opposite eigenvalues, it follows from a classical result of
Brjuno (see \cite{MR647488}), that $X$ is analytically orbitally
linearizable if and only if it is formally orbitally linearizable.
\item The fact of being orbitally linearizable is naturally invariant under
orbital equivalence, and then, by (almost) uniqueness of $c_{1},c_{2}$
in Theorem \ref{thm: uniqueness}, if $Y\in\snnd$ is asymptotically
linearizable, then its formal invariants $c_{1},c_{2}$ satisfy $c_{1}+c_{2}=0$.
In this case, we write $c:=c_{2}=-c_{1}$.
\end{enumerate}
\end{rem}
The two remarks above imply the following corollary.
\begin{cor}
\label{cor: normalisation asympt ham}Let $Y\in\snnd$ be a doubly-resonant
saddle-node asymptotically orbitally linearizable such that $Y_{0}:=Y_{\mid\acc{x=0}}$
be a germ of an analytic vector field in $\left(\mbox{\ensuremath{\ww C^{2}}},0\right)$.
Then, there exists $\Phi\in\fdiff$ such that $\Phi_{\mid\acc{x=0}}$
be a germ an analytic diffeomorphism in $\left(\ww C^{2},0\right)$
and: 
\begin{eqnarray*}
\Phi_{*}\left(Y\right) & = & x^{2}\pp x+\left(-\lambda+a_{1}x-c\left(v\right)\right)y_{1}\pp{y_{1}}+\left(\lambda+a_{2}x+c\left(v\right)\right)y_{2}\pp{y_{2}}\,\,\,\,,
\end{eqnarray*}
where we put $v:=y_{1}y_{2}$. Here, $c\left(v\right)\in v\ww C\left\{ v\right\} $
is a germ of an analytic function vanishing at the origin, and $a_{1},a_{2}\in\ww C$
are such that $a_{1}+a_{2}=\res Y$. Moreover, $\Phi$ is unique up
to linear transformations.
\end{cor}
It is important to notice that the following property holds.
\begin{prop}
\label{prop: trans ham implique asymp ham}If $Y\in\sns$ is doubly-resonant
transversally Hamiltonian saddle-node, then $Y$ is asymptotically
orbitally linearizable. \end{prop}
\begin{proof}
The facts that $\cal L_{Y}\left(\omega\right)\in\ps{\mbox{d}x}$ and
$\cal L_{Y}\left(x\right)=x^{2}$ imply that $a_{1}+a_{2}=1$ and
then: 
\begin{eqnarray*}
\cal L_{Y}\left(\tx dy_{1}\wedge\tx dy_{2}\right) & = & x\left(\tx dy_{1}\wedge\tx dy_{2}\right)+\ps{\mbox{d}x}\,\,\,\,.
\end{eqnarray*}
Consequently, if we denote $Y_{0}:=Y_{\mid x=0}$ the restriction
of $Y$ to the invariant hypersurface $\acc{x=0}$, we have: 
\[
\cal L_{Y_{0}}\left(\tx dy_{1}\wedge\tx dy_{2}\right)=0\,\,.
\]
This means that $Y_{0}$ is a Hamiltonian vector field, \emph{i.e.}
there exists $H\left(\mathbf{y}\right)\in\form{\mathbf{y}}$ such
that: 
\[
Y_{0}\left(y_{1},y_{2}\right)=-\ppp H{y_{2}}\left(y_{1},y_{2}\right)\pp{y_{1}}+\ppp H{y_{1}}\left(y_{1},y_{2}\right)\pp{y_{2}}\,\,.
\]
Possibly by performing a linear change of coordinate, we can assume
that $H\left(\mathbf{y}\right)\in\lambda y_{1}y_{2}+\mathfrak{m}^{3}$,
therefore we can write: 
\[
Y=x^{2}\pp x+\left(-\ppp H{y_{2}}+xF_{1}\left(x,\mathbf{y}\right)\right)\pp{y_{1}}+\left(\ppp H{y_{1}}+xF_{2}\left(x,\mathbf{y}\right)\right)\pp{y_{2}}\,\,\,,
\]
where $F_{1},F_{2}\in\form{x,\mathbf{y}}$ vanish at the origin. If
we define ${\displaystyle J:=\left(\begin{array}{cc}
0 & -1\\
1 & 0
\end{array}\right)\in M_{2}\left(\ww C\right)}$ and $\nabla H:=\,^{t}\left(\mbox{D}H\right)$, then ${\displaystyle Y_{\mid\acc{x=0}}=J\nabla H}$.
According to the Morse lemma for holomorphic functions, there exists
an analytic change of coordinates $\varphi\in\diff{}$ in $\left(\ww C^{2},0\right)$
tangent to the identity such that ${\displaystyle {\displaystyle \widetilde{H}\left(\mathbf{y}\right):=H\left(\varphi^{-1}\left(\mathbf{y}\right)\right)=y_{1}y_{2}}}$.
Let us now recall a trivial result from linear algebra.
\begin{fact*}
Let ${\displaystyle J:=\left(\begin{array}{cc}
0 & -1\\
1 & 0
\end{array}\right)\in M_{2}\left(\ww C\right)}$, and $P\in M_{2}\left(\ww C\right)$. Then, ${\displaystyle PJP^{\mathrm{t}}=\mbox{det}\left(P\right)J}$.
\end{fact*}
We deduce the next result. 
\begin{lem*}
\label{lem: hamiltonian system}Let $H\in\mathfrak{m}^{2}\subset\form{\mathbf{y}}$,
$Y_{0}:=J\nabla H$ the associated Hamiltonian vector field in $\ww C^{2}$
(for the standard symplectic form $dy_{1}\wedge dy_{2}$), and an
analytic diffeomorphism near the origin denoted by $\varphi$. Then:
\begin{eqnarray*}
\varphi_{*}\left(Y_{0}\right) & := & \left(\tx D\varphi\circ\varphi^{-1}\right)\cdot\left(Y_{0}\circ\varphi^{-1}\right)=\det\left(\tx D\varphi\circ\varphi^{-1}\right)J\nabla\widetilde{H}\,\,\,\,,
\end{eqnarray*}
where $\widetilde{H}:=H\circ\varphi^{-1}$.
\end{lem*}
As a conclusion,the previous lemma shows that $Y$ is asymptotically
orbitally linearizable.
\end{proof}
The next property is a straightforward consequence of Corollary \ref{cor: normalisation asympt ham},
Proposition \ref{prop: trans ham implique asymp ham} and Theorem
\ref{thm: Th ham}.
\begin{cor}
\label{cor: normalisation analytic symplectique}Let $Y\in\sns$ be
a transversally Hamiltonian doubly-resonant saddle-node. Then, there
exists a transversally symplectic diffeomorphism $\Phi\in\sdiff$
such that $\Phi_{\mid\acc{x=0}}$ be a germ an analytic diffeomorphism
in $\left(\ww C^{2},0\right)$ and: 
\begin{eqnarray}
\Phi_{*}\left(Y\right)= & x^{2}\pp x+\left(-\lambda+a_{1}x-c\left(v\right)\right)y_{1}\pp{y_{1}}+\left(\lambda+a_{2}x+c\left(v\right)\right)y_{2}\pp{y_{2}}\,\,\,\,.\label{eq: fibered normal form-1-1-1}
\end{eqnarray}
where we put $v:=y_{1}y_{2}$. Here, $c\left(v\right)\in v\ww C\left\{ v\right\} $
is a germ of an analytic function vanishing at the origin, and $a_{1},a_{2}\in\ww C$
are such that $a_{1}+a_{2}=\res Y=1$. Moreover, $\Phi$ is unique
up to linear symplectic transformations, and:
\begin{eqnarray*}
\left(\Phi_{\mid\acc{x=0}}\right)^{*}\left(\mbox{d}y_{1}\wedge\mbox{d}y_{2}\right) & = & \mbox{d}y_{1}\wedge\mbox{d}y_{2}~.
\end{eqnarray*}

\end{cor}

\subsection{Periods of the Hamiltonian on $\protect\acc{x=0}$}

From now on, we consider a vector field 
\[
Y=x^{2}\pp x+\left(\left(-\ppp H{y_{2}}+xF_{1}\left(x,\mathbf{y}\right)\right)\pp{y_{1}}+\left(\ppp H{y_{1}}+xF_{2}\left(x,\mathbf{y}\right)\right)\pp{y_{2}}\right)\,\,\,,
\]
with $H\left(\mathbf{y}\right)=\lambda y_{1}y_{2}+\underset{\mathbf{y}\rightarrow0}{O}\left(\left\Vert \mathbf{y}\right\Vert ^{3}\right)$
analytic at the origin of $\ww C^{2}$, and $F_{1},F_{2}\in\form{x,\mathbf{y}}$
vanishing at the origin. Let us consider the restriction $Y_{0}:=Y_{\mid\acc{x=0}}$:
it is an analytic Hamiltonian vector field in $\left(\ww C^{2},0\right)$:
\[
Y_{0}=-\ppp H{y_{2}}\pp{y_{1}}+\ppp H{y_{1}}\pp{y_{2}}\,\,.
\]
We fix a small polydisc $\mathbf{D}\left(0,\mathbf{r}\right)\subset\ww C^{2}$
on which $H$ is analytic with $\mathbf{r}=\left(r_{1},r_{2}\right)$.
The leaves of the foliation defined by $Y_{0}$ in $\mathbf{D}\left(0,\mathbf{r}\right)$
are given by the level curves $L_{a}:=\acc{H=a}\cap\mathbf{D}\left(0,\mathbf{r}\right)$,
$a\in\mbox{D}\left(0,r\right)$, with $r>0$ small enough. The Morse
Lemma for holomorphic functions tells us that $L_{a}$ is topologically
a cylinder for $a\neq0$, and $r$,$r_{1},r_{2}$ small enough. Thus
we can consider a generator $\gamma_{a}$ of the first homology group
of $L_{a}$. We also consider a time-form for $Y_{0}$, which is a
meromorphic $1$-form $\tau_{Y_{0}}$ in $\mathbf{D}\left(0,\mathbf{r}\right)$
with a unique pole at the origin and such that $\tau_{Y_{0}}\cdot\left(Y_{0}\right)=1$.
For instance, take ${\displaystyle \tau_{Y_{0}}=-\frac{\mbox{d}y_{1}}{\ppp H{y_{2}}}}$.

Now we define the associated period map: 
\begin{eqnarray*}
T_{H}~:~\mbox{D}\left(0,r\right)\backslash\acc 0 & \longrightarrow & \ww C\\
a & \longmapsto & T_{H}\left(a\right):=\frac{1}{2i\pi}\oint_{\gamma_{a}}\tau_{Y_{0}}\,\,.
\end{eqnarray*}
This mapping is a well-defined meromorphic function of $a\in\mbox{D}\left(0,r\right)$.
\begin{prop}
For $r>0$ small enough, and $a\in\mbox{D}\left(0,r\right)\backslash\acc 0$,
$T_{Y_{0}}\left(a\right)$ only depends on the class of $\gamma_{a}$
in $H_{1}\left(L_{a},\ww Z\right)$. In other words, if $\tau'_{Y_{0}}$
is another time-form of $Y_{0}$ and $\gamma'_{a}$ is any loop in
$L_{a}$ homologous to $\gamma_{a}$, then 
\[
\oint_{\gamma_{a}}\tau_{Y_{0}}=\oint_{\gamma'_{a}}\tau'_{Y_{0}}\,\,\,.
\]
\end{prop}
\begin{proof}
The fact that this quantity does not depend on a specific choice of
a representative of $\gamma_{a}$ in its homology class comes from
Stokes Theorem. The fact that it does not depend on the choice of
a specific time-form comes from the fact that $\gamma_{a}$ lies in
a leaf of the foliation generated by $Y_{0}$. If 
\begin{eqnarray*}
\gamma_{a}:\cro{0,1} & \rightarrow & L_{a}\\
t & \mapsto & \left(\gamma_{a,1}\left(t\right),\gamma_{a,2}\left(t\right)\right)\,\,\,\,,
\end{eqnarray*}
then ${\displaystyle \frac{\mbox{d}}{\mbox{d}t}\left(\gamma_{a}\right)\left(t\right)}=v_{a}\left(t\right)Y_{0}\left(\gamma_{a}\left(t\right)\right)$,
where 
\[
v_{a}\left(t\right)=\frac{1}{\left(-\ppp H{y_{2}}\left(\gamma_{a}\left(t\right)\right)\right)}\ddd{\gamma_{a,1}\left(t\right)}t=\frac{1}{\left(\ppp H{y_{1}}\left(\gamma_{a}\left(t\right)\right)\right)}\ddd{\gamma_{a,2}\left(t\right)}t\qquad.
\]
Then: 
\begin{eqnarray*}
\oint_{\gamma_{a}}\tau_{Y_{0}} & = & \int_{0}^{1}\tau_{Y_{0}}\left(\gamma_{a}\left(t\right)\right)\cdot\left({\displaystyle \frac{\mbox{d}}{\mbox{d}t}\left(\gamma_{a}\right)\left(t\right)}\right)\mbox{d}t\\
 & = & \int_{0}^{1}\tau_{Y_{0}}\left(\gamma_{a}\left(t\right)\right)\cdot\left(v_{a}\left(t\right)Y_{0}\left(\gamma_{a}\left(t\right)\right)\right)\mbox{d}t\\
 & = & \int_{0}^{1}v_{a}\left(t\right)\mbox{d}t
\end{eqnarray*}
since $\tau_{Y_{0}}\cdot\left(Y_{0}\right)=1$.\end{proof}
\begin{defn}
We call $T_{H}$ the period map of $H$ near the origin. 
\end{defn}
Now, consider a germ of an analytic diffeomorphism $\Psi$ fixing
the origin of $\ww C^{2}$. Then: 
\begin{eqnarray*}
T_{H}\left(a\right) & =\frac{1}{2i\pi} & \oint_{\gamma_{a}}\tau_{Y_{0}}\\
 & =\frac{1}{2i\pi} & \oint_{\Psi^{-1}\left(\gamma_{a}\right)}\Psi^{*}\left(\tau_{Y_{0}}\right)\,\,\,.
\end{eqnarray*}
Notice that if we write $X_{0}:=\left(\Psi^{^{-1}}\right)_{*}\left(Y_{0}\right)$
and $\tau_{X_{0}}:=\Psi^{*}\left(\tau_{Y_{0}}\right)$, then: 
\[
\tau_{X_{0}}\cdot\left(X_{0}\right)=\left(\Psi^{*}\left(\tau_{Y_{0}}\right)\right)\cdot\left(\left(\Psi^{-1}\right)_{*}\left(Y_{0}\right)\right)=\tau_{Y_{0}}\cdot\left(Y_{0}\right)=1\,\,.
\]
Now, let us take $\Psi^{-1}=\Phi_{\mid\acc{x=0}}$ as in Corollary
\ref{cor: normalisation analytic symplectique} such that 
\[
{\displaystyle X_{0}=\left(\lambda+c\left(v\right)\right)\left(-y_{1}\pp{y_{1}}+y_{2}\pp{y_{2}}\right)\,\,,}
\]
with $v=y_{1}y_{2}$ , $c\in\ww C\acc v$ and $c\left(0\right)=0$.
Then $\tilde{\gamma}_{a}:=\Psi^{-1}\left(\gamma_{a}\right)=\Phi_{\mid\acc{x=0}}\left(\gamma_{a}\right)$
is a loop generating the homology of the leaf $\Phi_{\mid\acc{x=0}}\left(L_{a}\right)$. 

Consider $h:=H\circ\Psi$ near the origin. Then, $\Phi_{\mid\acc{x=0}}\left(L_{a}\right)=\acc{h=a}$
in a neighborhood of the origin. Notice that $h$ depends in fact
only on $v=y_{1}y_{2}$, and $h\left(v\right)=\lambda v+\underset{\left|v\right|\rightarrow0}{o}\left(\abs v\right)$.
Since $\lambda\neq0$, the inverse function theorem ensures the existence
of an analytic function $g\in\ww C\acc v$ such that $g\left(0\right)=0$
and $h\circ g\left(v\right)=g\circ h\left(v\right)=v$ in a neighborhood
of $0$. Thus, $\acc{h\left(v\right)=a}=\acc{v=g\left(a\right)}$.
Consequently, taking for instance ${\displaystyle \tau_{X_{0}}=-\frac{\mbox{d}y_{1}}{y_{1}\left(\lambda+c\left(v\right)\right)}}$,
we see that: 
\begin{eqnarray*}
T_{H}\left(a\right) & =\frac{1}{2i\pi} & \oint_{\tilde{\gamma}_{a}}\tau_{X_{0}}\\
 & = & \frac{1}{2i\pi}\frac{1}{\lambda+c\left(g\left(a\right)\right)}\oint_{\tilde{\gamma}_{a}}-\frac{\mbox{d}y_{1}}{y_{1}}\\
 & = & \frac{-1}{\lambda+c\left(g\left(a\right)\right)}
\end{eqnarray*}
according to the orientation chosen for $\gamma_{a}$. In particular,
we see that $T_{H}$ is analytic at the origin, and $T_{H}$ can be
extend at $0$ by $\frac{-1}{\lambda}$.

The fact that $\Phi_{\mid\acc{x=0}}$ satisfies 
\[
\left(\Phi_{\mid\acc{x=0}}\right)^{*}\left(\mbox{d}y_{1}\wedge\mbox{d}y_{2}\right)=\mbox{d}y_{1}\wedge\mbox{d}y_{2}
\]
 implies that $\det\left(\Phi_{\mid\acc{x=0}}\right)=1$, so that
\[
X_{0}=-\ppp h{y_{2}}\pp{y_{1}}+\ppp h{y_{1}}\pp{y_{2}}
\]
and
\[
\ddd hv=\lambda+c\left(v\right)\,\,\,.
\]
As a consequence, for $a=h\left(v\right)$, we have the following
relation: 
\[
\ddd hv\left(v\right).T_{H}\left(h\left(v\right)\right)=-1\,\,\,\,.
\]
If we consider the antiderivative $S_{H}$ of $T_{H}$ such that $S_{H}\left(0\right)=0$,
we have 
\[
S_{H}\left(h\left(v\right)\right)=-v\,\,\,,
\]
and in particular 
\[
S_{H}=-g\,\,.
\]

Let us summarize this study in the following proposition.
\begin{prop}
\label{prop: periode}Let 
\[
Y_{0}=-\ppp H{y_{2}}\pp{y_{1}}+\ppp H{y_{1}}\pp{y_{2}}
\]
be the restriction on $\acc{x=0}$ of a transversally Hamiltonian
doubly-resonant saddle-node $Y\in\sns$, where $H\left(\mathbf{y}\right)=\lambda y_{1}y_{2}+\underset{\mathbf{z}\rightarrow0}{o}\left(\left\Vert \mathbf{y}\right\Vert ^{2}\right)$
is analytic at the origin of $\ww C^{2}$. Consider its unique transversally
Hamiltonian normal form 
\[
X=x^{2}\pp x+\left(-\left(\lambda+c\left(v\right)\right)+a_{1}x\right)y_{1}\pp{y_{1}}+\left(\lambda+c\left(v\right)+a_{2}x\right)y_{2}\pp{y_{2}}
\]
given by Theorem \ref{thm: Th ham}. Consider the period map $T_{H}$
as defined above. Then the following holds: 
\begin{enumerate}
\item $c$ is the germ of an analytic function at the origin.
\item $T_{H}$ defines the germ of an analytic function in a neighborhood
of $0\in\ww C^{2}$, such that $T_{H}\left(0\right)=\frac{-1}{\lambda}$.
\item If $S_{H}$ is the primitive of $T_{H}$ such that $S_{H}\left(0\right)=0$,
then $\left(-S_{H}\right)$ is invertible (for the composition), and
its inverse $h$ satisfy: 
\[
\ddd hv\left(v\right)=\lambda+c\left(v\right)\,\,\,.
\]

\end{enumerate}
\end{prop}
The conclusion is that if one is able to compute the period map of
the original Hamiltonian vector field on $\acc{x=0}$, then one can
compute the formal invariant $c$ in the normal form given in Theorem
\ref{thm: Th ham}, which is in fact even analytic in this case.
\begin{rem}
The Hamiltonian function $h\left(y_{1}y_{2}\right)=\lambda y_{1}y_{2}+\intop^{y_{1}y_{2}}c\left(v\right)\mbox{d}v$
is in fact the \emph{symplectic normal form} of the original Hamiltonian
function $H\left(y_{1},y_{2}\right)=\lambda y_{1}y_{2}+\underset{\mathbf{z}\rightarrow0}{o}\left(\left\Vert \mathbf{y}\right\Vert ^{2}\right)$,
as described in \cite{Chow} (section $2.7$).
\end{rem}

\subsection{Example: the case of the first Painlevé equation}

~

In the case of the first Painlevé equation, in appropriate coordinates,
we are working with the Hamiltonian 
\[
H\left(y_{1},y_{2}\right)=\frac{1}{5}\left(-2y_{2}^{2}+24\zeta y_{1}^{2}+8y_{1}^{3}\right)
\]
where $\zeta=\frac{i}{\sqrt{6}}$, according to equation (\ref{eq: P1 at infinity}).
These are not the system of coordinates which diagonalizes the linear
part of the vector field, but the value of the period does not changes
by symplectic changes of coordinates (those which preserve $\tx dy_{1}\wedge\tx dy_{2}$).
Now, if we fix $a\neq0$ with $\abs a$ small enough and look at the
level curve $\acc{H=a}$ near the origin in $\ww C^{2}$, we can compute
the associated period: 
\begin{eqnarray*}
T_{H}\left(a\right) & = & \frac{1}{2i\pi}\oint_{\gamma_{a}}\frac{5\mbox{d}y_{1}}{-4y_{2}}\\
 & = & {\displaystyle \frac{1}{2i\pi}\oint_{\gamma_{a,1}}\frac{5\mbox{d}y_{1}}{-4\sqrt{12\zeta y_{1}^{2}+4y_{1}^{3}-\frac{5}{2}a}}}\,\,\,\,,
\end{eqnarray*}
where $\gamma_{a,1}$ is the component of $\gamma_{a}$ with respect
to $\pp{y_{1}}$. 
\begin{rem}
The period $T_{H}\left(a\right)$ is one of the periods of the Weierstrass
function $\wp$ associated to the cubic 
\begin{eqnarray*}
H\left(y_{1},y_{2}\right) & = & a
\end{eqnarray*}
(see \emph{e.g. }\cite{Brjuno&Goryuchkina}, \cite{Abenda&Bazzani}).
To compute it we can chose for instance 
\begin{eqnarray*}
\gamma_{a,1}:\cro{0,2\pi} & \longrightarrow & \ww C\\
t & \longmapsto & \rho_{a}e^{it}
\end{eqnarray*}
where $\rho_{a}>0$ is such that $\abs{12\zeta y_{1}^{2}+4y_{1}^{3}}>\frac{5\abs a}{2}$,
for all $y_{1}=\gamma_{a,1}\left(t\right)$, $t\in\cro{0,2\pi}$.
Now we write:
\begin{eqnarray*}
\frac{1}{-4\sqrt{12\zeta y_{1}^{2}+4y_{1}^{3}-\frac{5}{2}a}} & = & \frac{\sqrt{2}}{-4\sqrt{24\zeta y_{1}^{2}+8y_{1}^{3}}}\frac{1}{\sqrt{1-\frac{5a}{24\zeta y_{1}^{2}+8y_{1}^{3}}}}\\
 & = & \frac{\sqrt{2}}{-4\sqrt{24\zeta y_{1}^{2}+8y_{1}^{3}}}\sum_{k\geq0}\left(\begin{array}{c}
\frac{-1}{2}\\
k
\end{array}\right)\left(\frac{5a}{24\zeta y_{1}^{2}+8y_{1}^{3}}\right)^{k}\,\,\,.
\end{eqnarray*}
As we have normal convergence, we can swap the order of summation
and integration: 
\begin{eqnarray*}
T_{H}\left(a\right) & = & \frac{5}{2i\pi}\oint_{\gamma_{a,1}}\frac{\sqrt{2}}{-4\sqrt{24\zeta y_{1}^{2}+8y_{1}^{3}}}\sum_{k\geq0}\left(\begin{array}{c}
\frac{-1}{2}\\
k
\end{array}\right)\left(\frac{5a}{24\zeta y_{1}^{2}+8y_{1}^{3}}\right)^{k}\mbox{d}y_{1}\\
 & = & \frac{-5\sqrt{2}}{8i\pi}\sum_{k\geq0}\left(\begin{array}{c}
\frac{-1}{2}\\
k
\end{array}\right)5^{k}\left(\oint_{\gamma_{a,1}}\left(24\zeta y_{1}^{2}+8y_{1}^{3}\right)^{-\left(k+\frac{1}{2}\right)}\mbox{d}y_{1}\right)a^{k}\,\,\,\,.
\end{eqnarray*}
Notice that $y_{1}\mapsto\left(24\zeta y_{1}^{2}+8y_{1}^{3}\right)^{\left(k+\frac{1}{2}\right)}$
is in fact analytic in a neighborhood of the origin, with a zero of
order $2k+1$. Hence we can compute the integral above using the residue
theorem. As we have
\end{rem}
\begin{eqnarray*}
\left(24\zeta y_{1}^{2}+8y_{1}^{3}\right)^{-\left(k+\frac{1}{2}\right)} & = & \left(24\zeta\right)^{-\left(k+\frac{1}{2}\right)}y_{1}^{-\left(2k+1\right)}\sum_{j\geq0}\left(\begin{array}{c}
-\left(k+\frac{1}{2}\right)\\
j
\end{array}\right)\left(\frac{8}{24\zeta}\right)^{j}y_{1}^{j}\qquad,
\end{eqnarray*}
we see that the associated residue at $0$ is equal to $8^{2k}{\displaystyle \left(24\zeta\right)^{-\left(3k+\frac{1}{2}\right)}}\left(\begin{array}{c}
-\left(k+\frac{1}{2}\right)\\
2k
\end{array}\right)$, so that 
\[
T_{H}\left(a\right)={\displaystyle \sum_{k\geq0}}T_{H,k}a^{k}
\]
 with: 
\[
T_{H,k}=-5^{k+1}\left(\begin{array}{c}
\frac{-1}{2}\\
k
\end{array}\right)\left(\begin{array}{c}
-\left(k+\frac{1}{2}\right)\\
2k
\end{array}\right)8^{-\left(k+1\right)}{\displaystyle \left(3\zeta\right)^{-\left(3k+\frac{1}{2}\right)}}\,\,\,\,.
\]
Using notations of Proposition \ref{prop: periode} we have: 
\[
S_{H}\left(a\right)={\displaystyle \sum_{k\geq0}}\frac{T_{H,k}}{k+1}a^{k+1}={\displaystyle \sum_{k\geq1}}S_{H,k}a^{k}\,\,\,\:
\]
with $S_{H,k}=\frac{T_{H,k-1}}{k}$ for $k\geq1$. Since $S_{H}\left(0\right)=0$
and $\mbox{\ensuremath{\ddd{S_{H}}a\left(0\right)}}=T_{H}\left(0\right)\neq0$,
the mapping $\left(-S_{H}\right)$ is invertible for the composition
and we can compute recursively its inverse (denoted by $h$): 
\[
h\left(v\right)=\sum_{k\geq1}h_{k}v^{k}\,\,\,\,.
\]
For all $k\geq1$, the coefficient $h_{k}$ is uniquely determined
by the coefficients $S_{H,j},\, j\leq k$. Finally, we have 
\[
\lambda+{\displaystyle c\left(v\right)=\ddd hv\left(v\right)=\sum_{k\geq0}\left(k+1\right)h_{k+1}v^{k}=\lambda+\sum_{k\geq1}c_{k}v^{k}}\,\,.
\]

As a conclusion, the jet of order $k$ of $T_{H}$ gives us the jet
of order $k$ of $c$. After computations performed with Maple, we
obtain for instance: 
\begin{eqnarray*}
\lambda & = & \frac{8\sqrt{3\zeta}}{5}=\frac{4\cdot2^{\frac{3}{4}}\cdot3^{\frac{1}{4}}}{5}e^{\frac{i\pi}{4}}\\
c_{1} & = & 3\\
c_{2} & = & 9+\frac{167\cdot2^{\frac{1}{4}}\cdot3^{\frac{3}{4}}}{96}e^{\frac{3i\pi}{4}}\\
c_{3} & = & 16+\frac{31837\sqrt{6}}{6912}i+\frac{5}{2}\cdot2^{\frac{1}{4}}\cdot3^{\frac{1}{4}}\cdot e^{\frac{3i\pi}{4}}\,\,\,\,.
\end{eqnarray*}
One can in fact compute any finite jet of $c$.
\begin{rem}
Similar computations can be performed for any Hamiltonian of the form
$H\left(y_{1},y_{2}\right)=\beta y_{2}^{2}+\alpha y_{1}^{2}+f\left(y_{1}\right)$,
where $\alpha,\beta\in\ww C\backslash\acc 0$ and $f\in\ww C\acc{y_{1}}$. 
\end{rem}

\section{Acknowledgement}

I would like to express my gratitude to my supervisors Daniel Panazzolo
and Loïc Teyssier for their useful comments and suggestions. The many
discussions we had played an important role in the improvement of
this paper.

\section*{References}

\bibliographystyle{plain}
\addcontentsline{toc}{section}{\refname}\bibliography{references}

\end{document}